\documentclass[pdflatex,12pt,a4paper,twoside]{article}
\usepackage{authblk}
\usepackage{amsmath,amsfonts,amssymb,amsthm}
\usepackage{ifthen}
\usepackage{graphicx}
\usepackage{epsfig}
\usepackage{nicefrac}
\usepackage{mathrsfs}
\usepackage[a4paper,left=31mm,right=31mm,top=29mm,bottom=26mm]{geometry}
\usepackage{amsmath}
\usepackage{amsfonts}
\newcommand{\R}{\mathbb{R}}
\newcommand{\N}{\mathbb{N}}
\newcommand{\C}{\mathbb{C}}

\newcommand{\eps}{\varepsilon}
\newcommand{\fhi}{\varphi}

\newcommand{\weakto}{\rightharpoonup}

\newcommand{\ra}{\rangle}
\newcommand{\la}{\langle}

\newcommand{\del}{\partial}

\newcommand{\id}{\mathrm{id}}

\newcommand{\per}{\mathrm{per}}

\def\ri{{\rm i}}%

\renewcommand{\Re}{\mathrm{Re}}

\def\XXint#1#2#3{{\setbox0=\hbox{$#1{#2#3}{\int}$}
     \vcenter{\hbox{$#2#3$}}\kern-.5\wd0}}

\usepackage[
    bookmarks,
    colorlinks=true,
    bookmarksopen=true,
    bookmarksopenlevel=1,
    linkcolor=black,
    citecolor=blue,
    urlcolor=blue,
    filecolor=blue,
    anchorcolor=red,
    pdfstartview=FitH,
    pdfpagelayout=OneColumn,
    plainpages=false, 
    pdfpagelabels, 
    hypertexnames=false, 
    linktocpage, 
]{hyperref}

\newtheorem{theorem}{Theorem}[section]

\newtheorem{lemma}[theorem]{Lemma}
\newtheorem{proposition}[theorem]{Proposition}

\newtheorem{assumption}[theorem]{Assumption}

\newtheorem{remark}[theorem]{Remark}

\newtheorem{algorithm}{Algorithm}

\numberwithin{equation}{section}

\title{Dispersive homogenized models and coefficient formulas for
  waves in general periodic media}

\author{T.\,Dohnal, A.\,Lamacz, B.\,Schweizer\thanks{Technische
    Universit\"at Dortmund, Fakult\"at f\"ur Mathematik, Vogelpothsweg
    87, D-44227 Dortmund, Germany.}}

\begin{document}

\maketitle

\begin{abstract}
  We analyze a homogenization limit for the linear wave equation of
  second order.  The spatial operator is assumed to be of divergence
  form with an oscillatory coefficient matrix $a^\eps$ that is
  periodic with characteristic length scale $\eps$; no spatial
  symmetry properties are imposed. Classical homogenization theory
  allows to describe solutions $u^\eps$ well by a non-dispersive wave
  equation on fixed time intervals $(0,T)$. Instead, when larger time
  intervals are considered, dispersive effects are observed. In this
  contribution we present a well-posed weakly dispersive equation with
  homogeneous coefficients such that its solutions $w^\eps$ describe
  $u^\eps$ well on time intervals $(0,T\eps^{-2})$. More precisely, we
  provide a norm and uniform error estimates of the form $\| u^\eps(t)
  - w^\eps(t) \| \le C\eps$ for $t\in (0,T\eps^{-2})$. They are
  accompanied by computable formulas for all coefficients in the
  effective models. We additionally provide an $\eps$-independent
  equation of third order that describes dispersion along rays and we
  present numerical examples.
\end{abstract}

\medskip {\bf Keywords:} wave equation, large time homogenization,
dispersive model, Bloch analysis

  \medskip
  {\bf MSC:} 35B27, 35L05

\pagestyle{myheadings} 
\thispagestyle{plain} 

\markboth{T.\,Dohnal, A.\,Lamacz, B.\,Schweizer}{Dispersive
  homogenized models for waves in general periodic media}

\section{Introduction}

Waves in heterogeneous media exhibit dispersion. This fact is
well-known in physics and it can be observed also for waves that are
described (microscopically) by the classical, non-dispersive wave
equation. Our aim in this contribution is to cast the effect in
mathematical terms, to present a well-posed, dispersive effective wave
equation, and to provide computable formulas for the (homogeneous)
coefficients in the effective equation.

Our analysis concerns solutions $u^\eps: \R^n\times (0,\infty)\to \R$,
$n\in \{1,2,3\}$, of the linear wave equation in periodic media,
\begin{equation}
  \label{eq:eps-wave}
  \del_t^2 u^\eps(x,t) = \nabla\cdot (a^\eps(x) \nabla u^\eps(x,t))\,.
\end{equation}
The medium is characterized by a positive, symmetric coefficient
matrix field $a^\eps :\R^n \to \R^{n\times n}$.  We are interested in
periodic media with a small periodicity length-scale $\eps>0$, and
assume that $a^\eps(x) = a_Y(x/\eps)$ where $a_Y:\R^n\to \R^{n\times
  n}$ is periodic with the periodicity of the unit cell $Y =
(-\pi,\pi)^n$. Except for positivity, matrix symmetry, and
periodicity, no assumptions on $a_Y(.)$ are made (in contrast to our
earlier paper \cite{Bloch-DLS-2013}, where certain spatial symmetries
are exploited). Our interest is to describe the solutions $u^\eps$ for
large times, $t\sim \eps^{-2}$. For classical homogenization results
(derivation of effective equations on fixed time intervals) we refer
to \cite {FrancfortMR1172450, FrancfortMurat1992} and mention here
that, due to energy conservation, even classical homogenization
results for the wave equation are much more involved than
corresponding results e.g.~for the heat equation (the ``intermediate
case'', the wave equation with damping is considered in \cite
{OriveZuazua}). To simplify the exposition, we work here with smooth
coefficients $a^\eps$, noting that the regularity of the coefficient
is crucial in observability results, see \cite{MR1921162}.

In order to have a well-defined object $u^\eps$, we must complement
the wave equation with an initial condition.  For notational
convenience, we restrict our analysis to a vanishing initial velocity,
i.e. to initial data
\begin{equation}
  \label{eq:initial}
  u^\eps(x,0) = f(x), \quad \del_t u^\eps(x,0) = 0\,.
\end{equation}
In our mathematical results, we will assume smoothness of $f$. More
precisely, we assume that $f\in L^2(\R^n)\cap L^1(\R^n)$ has the
Fourier representation
\begin{equation}
  \label{eq:f-cond}
  f(x) = \frac1{(2\pi)^{n/2}} \int_{\R^n} F_0(k)\, e^{+\ri k\cdot x}\, dk\,,
\end{equation}
where $F_0:\R^n\to \C$ has compact support $K\subset \R^n$.

We note that our assumptions imply the smoothness $f\in
C^\infty(\R^n)$. Less regular initial data can be treated with the
help of our results, exploiting the linearity of the equations:
Decomposing initial data with bounded energy into two parts, our
results can be applied to the smooth part, while the other part
generates an error that is, for all times, small in energy norm.

\subsection*{Known results on dispersive models}

The contribution \cite {ChenFishMR1896976} started a series of
articles \cite {ChenFishMR2097759, ChenFish-Uniformly,
  ChenFishMR1896977, ChenFishMR1896976} which is concerned with the
derivation of dispersive models for the wave equation. The authors
perform asymptotic (two-scale) expansions of $u^\eps$ in $\eps$ and
obtain with their formal calculations a fourth order equation of the
form
\begin{equation}
  \label{eq:formal-weakly-disp}
  \del_t^2 U^\eps = A D^2 U^\eps - \eps^2 C D^4 U^\eps \,,
\end{equation}
where $A$ and $C$ are homogeneous coefficients and $D$ denotes spatial
derivatives. They call this equation ``bad Boussinesq equation'', a
well-chosen name, considering the fact that the equation is ill-posed
(in the homogenization process, a positive matrix $A$ and a
non-positive tensor $C$ appear). We note that in the earlier article
\cite {SanSym} this equation also appears (with a sign typo) as a
result of a Bloch analysis, but it is not further analyzed in \cite
{SanSym}.  In \cite{ChenFishMR2097759, ChenFish-Uniformly,
  ChenFishMR1896977, ChenFishMR1896976} various approaches for a
further (analytical and numerical) exploitation of equation \eqref
{eq:formal-weakly-disp} are investigated: regularizations, non-local
approximations, and multiple time scales.

The first rigorous result that establishes a dispersive model for the
wave equation \eqref {eq:eps-wave} appeared in \cite {Lamacz-Disp}. In
that work, which is concerned with the one-dimensional case, the
well-posed dispersive equation \eqref {eq:weakly-dispersive} below is
formulated and an error estimate similar to \eqref {eq:approx-result}
is derived. The method of proof is very different from our approach
here (which is as in \cite{Bloch-DLS-2013}): Adaption operators are
constructed and used to adapt smooth solutions of the homogeneous
dispersive system to the periodic medium. After the adaption, direct
energy procedures can be applied.

Another mathematical derivation of dispersive limits is performed in
\cite {Allaire-PR-Diffractive2009, Allaire-PR-Diffractive2011}.  The
wave equation is scaled as in our setting (time scales of order
$\eps^{-2}$ are investigated), but the initial data are assumed to be
oscillatory at scale $\eps$ and are described by Bloch wave packets.
In this setting, the effective diffraction can be described by a
Schr\"odinger equation for the envelope function. Another scaling of
the system is analyzed in \cite {Allaire-Dispersive2003}, where large
potentials instead of large time spans are considered.

\subsubsection*{Bloch analysis} The central tool in a Bloch analysis
is the Bloch expansion of an arbitrary function (in our case the
solution $u^\eps$). While in a Fourier expansion one uses the dual
variable $k\in \R^n$, the Bloch expansion uses two dual variables, $k$
and $m$. Since $m\in \N_0$ is an additional parameter, the other
parameter varies only in a restricted domain, the Brillouin zone,
$k\in Z := (-1/2, 1/2)^n$.  The basis functions $e^{\ri k\cdot x}$ of
the Fourier analysis are replaced by solutions $\psi_m(.,k)$ of the
Bloch eigenvalue problem
\begin{equation}
  \label{eq:eigenvalue}
  - (\nabla_y + \ri k ) \cdot (a_Y(y) (\nabla_y + \ri k ) \psi_m(y,k) ) 
  = \mu_m(k) \psi_m(y,k)\,.
\end{equation}
Here $\psi_m(.,k) : Y \to \C$ is a periodic function, $\psi_m(.,k) \in
H^1_\per(Y)$, $0 \le \mu_0(k) \le \mu_1(k) \le \hdots$ are the
ordered, real eigenvalues.

Bloch wave homogenization theory establishes that the effective
behavior of $u^\eps$ in the limit $\eps\to 0$ is characterized solely
by the behavior of the smallest eigenvalue $\mu_0(k)$ in a
neighborhood of $k = 0\in Z$. For such results in classical
homogenization settings, we refer to \cite {Allaire-Bloch1998,
  MR1897707, MR2124894, MR2219790, MR1484944}.  In
Fig.~\ref{F:Bloch_wvs} we plot the Bloch wave functions $\psi_0(y,k)$
for $k=(1/2,0)$ and $k=(1/2,1/2)$. For an illustration of the
eigenvalue structure see Fig.~\ref{F:eps_conv} (a).
\begin{figure}[h!]%
  \begin{center}
    \includegraphics[width=7.5cm]{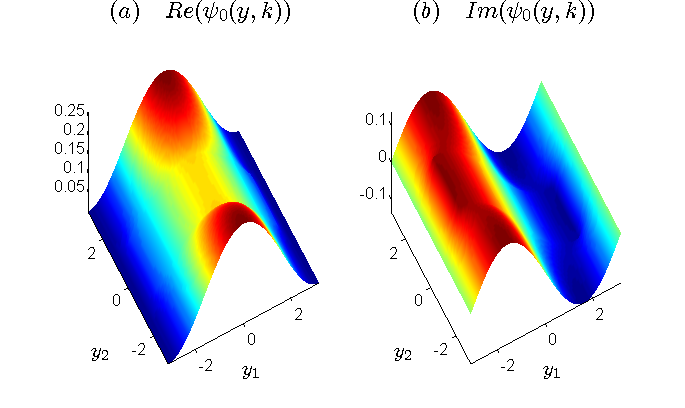}%
    \includegraphics[width=7.5cm]{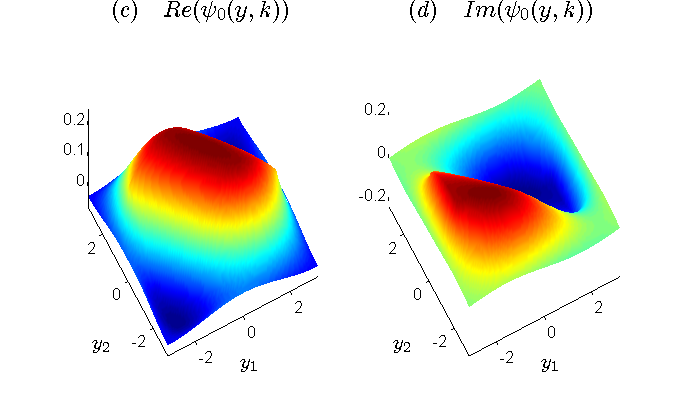}
    \caption{The Bloch wave $\psi_0(y,k)$ at $k=(1/2,0)$ in (a) and
      (b) and at $k=(1/2,1/2)$ in (c) and (d) corresponding to $\mu_0$
      for $a_Y$ from \eqref{E:a_rect}.}%
  \end{center}
  \label{F:Bloch_wvs}%
\end{figure}

The Bloch wave homogenization method was used for an analysis of
higher order effects of the heterogeneity of the medium in the
influential article \cite {SanSym}. That article does not formulate a
well-posed dispersive effective equation (hence, in particular, it
does not provide an error estimate), but it gives a lot of insight
into the dispersive limit: Even the effective {\em long time} behavior
of $u^\eps$ is characterized by $\mu_0(k)$ and its behavior near
$k=0$.  Expanding $\mu_0(k)$ in a Taylor series around $k = 0$, we may
write
\begin{equation}
  \label{eq:Taylor-mu}
  \mu_0(k)  =  \sum A_{lm} k_l k_m + \sum C_{lmnq} k_l k_m k_n k_q + O(|k|^6)\,,
\end{equation}
where odd derivatives vanish due to the symmetry $\mu_0(k)=\mu_0(-k)$,
sums are over repeated indices.  The matrix $A$ and the tensor $C$
provide the coefficients in the formal equation \eqref
{eq:formal-weakly-disp}, where $CD^4$ is the spatial fourth order
operator
\begin{equation}
  \label{eq:operator-C}
  C D^4 = \sum C_{ijkl} \del_i \del_j \del_k \del_l\,.
\end{equation}
While $A$ is positive definite and symmetric, $C$ turns out to be
negative semi-definite, a fact that is shown in \cite {MR2219790}. As
a consequence, the differential operator $AD^2$ is negative and the
operator $-\eps^2 CD^4$ is non-negative. For this reason, equation
\eqref {eq:formal-weakly-disp} is ill-posed.

Let us be more precise about the arguments in the Bloch analysis: We
start with the Bloch expansion of the solution $u^\eps$.  Using the
coefficients $\hat f_m^\eps(k)$ of a Bloch expansion of the initial
values $f$ and the Bloch eigenfunctions $\tilde\psi_m$ that have
$L^2(Y)$-norm $1$ we may write
\begin{equation}
  \label{eq:Bloch-u-eps}
  u^\eps(x,t) = \sum_{m=0}^\infty \int_{Z/\eps} \hat f_m^\eps(k) \tilde\psi_m(x/\eps,\eps k)
  e^{\ri k\cdot x}\, \Re\left(e^{\ri t \sqrt{\mu_m(\eps k)} / \eps}\right)\, dk\,.
\end{equation}
For a justification, see Lemma 2.1 of \cite{Bloch-DLS-2013}. In the
next steps, this formula is simplified for small $\eps>0$: One
realizes, to leading order in $\eps$, that only $m=0$ has to be
considered, that $\hat f_m^\eps$ can be replaced by the Fourier
transform $F_0$ of the initial values, and that $\tilde\psi_0$ can be
replaced by the constant $(2\pi)^{-n/2}$. Expanding finally
$\mu_0(\eps k)$ in $\eps$, one finds the following expression, which
can be used to define an approximate solution $v^\eps$.
\begin{equation}
  \label{eq:v-eps}
  \begin{split}
    v^\eps(x,t) :=  (2\pi)^{-n/2} \frac12 \sum_{\pm} \int_{K} F_0(k) 
    e^{\ri k\cdot x}
    \exp\left( \pm \ri t \sqrt{\sum A_{lm} k_l k_m}  \right)\qquad\qquad\\
    \times\   \exp\left( \pm \frac{\ri \eps^2}{2} t
      \frac{\sum C_{lmnq} k_l k_m k_n k_q}{\sqrt{\sum A_{lm} k_l k_m}}\right)\, dk\,.
  \end{split}
\end{equation}
The equation \eqref {eq:formal-weakly-disp} with tensors $A$ and $C$
is constructed in such a way that (formally) the function $v^\eps$ is
a solution up to errors of order $\eps^4$.

\subsubsection*{Rigorous approximation results} 
A rigorous mathematical analysis can be performed when the equation
\eqref {eq:formal-weakly-disp} is transformed into a well-posed
equation, using the replacement $A D^2 U^\eps\approx \del_t^2 U^\eps$
to re-write the operator $CD^4$. The first utilization of this trick
for a rigorous result seems to be in the treatment \cite {Lamacz-Disp}
in the one-dimensional case. More recently, we were able to exploit
the same trick in arbitrary dimension in \cite{Bloch-DLS-2013} (under
quite strong spatial symmetry assumptions on the coefficient field
$a_Y(.)$).

In that contribution, we use a rigorous Bloch wave analysis to show
that $v^\eps$ of \eqref {eq:v-eps} approximates $u^\eps$ in
appropriate norms.  In a second step, we show with energy methods that
$v^\eps$ is close to the solution $w^\eps$ of the well-posed, weakly
dispersive equation
\begin{equation}
  \label{eq:weakly-dispersive}
  \del_t^2 w^\eps = A D^2 w^\eps + \eps^2 E D^2 \del_t^2 w^\eps 
  - \eps^2 F D^4 w^\eps\,.
\end{equation}
The positive semi-definite and symmetric tensors $E$ and $F$ are
constructed in such a way that
\begin{equation}
  \label {eq:def-coeff-EF}
  -C D^4 = E D^2 A D^2- F D^4\,.
\end{equation}
Together, the two estimates provide an estimate for $u^\eps - w^\eps$.
This shows that the weakly dispersive equation \eqref
{eq:weakly-dispersive} is a valid replacement for the original
equation \eqref {eq:eps-wave} on large time intervals.

\subsection*{New results}
In the article at hand, we obtain the long-time homogenization result
for very general coefficient fields $a_Y(.)$. In particular, we show
that the well-posed equation \eqref {eq:weakly-dispersive} provides
the effective description of solutions for large times, characterizing
dispersion in arbitrary dimension and without spatial symmetry
assumptions. Furthermore, we provide explicit formulas for the
effective coefficients.

\paragraph{Decomposition lemma and approximation result.}
In order to show the approximation result, we can rely on the Bloch
wave analysis of \cite{Bloch-DLS-2013}. The only new ingredient is a
considerably developed decomposition lemma: Lemma \ref {lem:decompose}
below yields that, without any structural assumptions on $C$, the
differential operator $C D^4$ can be written as in \eqref
{eq:def-coeff-EF} for appropriate semi-definite and symmetric tensors
$E$ and $F$ (using the given positive, symmetric matrix $A$). In
particular, the lemma allows to decompose the operator $C D^4$ also
when the coefficients $a_Y(.)$ have no spatial symmetries.

Once the decomposition lemma is established, we can apply results of
\cite {Bloch-DLS-2013}. We obtain that \eqref {eq:weakly-dispersive}
is well-posed and that solutions $w^\eps$ approximate the solutions
$u^\eps$. More precisely, from the analysis of \cite {Bloch-DLS-2013},
we obtain an error estimate of the form $\| u^\eps(t) - w^\eps(t) \|
\le C_0\eps$ uniformly in $t\in (0,T\eps^{-2})$ in general periodic
media. This approximation result is stated and proved in Section \ref
{sec.mainproof}.

\paragraph{An $\eps$-independent third order dispersive equation.}
Our aim in Section \ref {sec.KdVrays} is to provide a simplified model
in which no $\eps$-dependence occurs.  The approximation result of
Theorem \ref {thm:main} below allows to analyze, instead of the
solution $u^\eps$ of the original problem, the solution $w^\eps$ of
the weakly dispersive equation \eqref {eq:weakly-dispersive}. In
Section \ref {sec.KdVrays} we analyze $w^\eps$ in two dimensions in
polar coordinates.  On every ray through the origin, for an
appropriate scaling of the solution, we can perform the limit $\eps\to
0$. The result is an equation that determines the shape of pulses,
namely a linear third order equation (a linearized KdV-equation). All
coefficients in this equation are computable from the coefficient
field $a_Y$.

\paragraph{Algorithms to compute effective quantities.}
In Section \ref {sec.Calculation} we present an algorithm that
provides the homogeneous coefficients in all effective equations,
i.e. $A$ and $C$ (in a more direct form than as derivatives of the
Bloch eigenvalue), the coefficients $E$ and $F$, and the coefficients
of the linearized KdV-equation.

The numerical results of Section \ref {sec.Calculation} compare
solutions to the original problem with solutions to the effective
problems. We find a remarkable qualitative and quantitative agreement
also for moderate $\eps$, and the correct experimental convergence
rates for $\eps\to 0$.

\section{A weakly dispersive effective equation}
\label{sec.mainproof}

In Section \ref {ssec.mainresult} we formulate our main result. The
main part of its proof can be obtained by applying results of
\cite{Bloch-DLS-2013}; this is the subject of Subsection \ref
{ssec.errorestimate}. The new ingredient is the decomposition lemma,
which is shown in Subsection \ref {ssec.decomposition}.

\subsection{Main approximation result}
\label{ssec.mainresult}

We emphasize that the set $Y\subset \R^n$, the reciprocal cell $Z :=
(-1/2, 1/2)^n\subset \R^n$, and the support $K\subset \R^n$ are fixed
data of the problem.  Given is also the coefficient field $a_Y$ that
determines $a^\eps(x) = a_Y(x/\eps)$.

\begin{assumption}\label{ass:a-f}
  The coefficient field $a_Y:\R^n\to \R^{n\times n}$ is $Y$-periodic
  for the cube $Y := (-\pi, \pi)^n \subset \R^n$ and has the
  regularity $a_Y\in C^1(\R^n,\R^{n\times n})$. The matrix $a_Y(y)$ is
  symmetric for every point $y\in \R^n$, i.e. $(a_Y(y))_{ij} =
  (a_Y(y))_{ji}$ for all $i,j\in\{1,...,n\}$. The field is positive
  definite: for some $\gamma>0$ there holds $\sum_{i,j=1}^n
  (a_Y(y))_{ij}\xi_i\xi_j\geq \gamma |\xi|^2$ for every $y\in\R^n$ and
  every $\xi\in\R^n$.
\end{assumption}

Our main approximation result is stated in the following theorem. The
result is very similar to the main theorem in
\cite{Bloch-DLS-2013}. The difference is that we do not assume any
spatial symmetry of $a_Y(.)$ (such as a reflection symmetry in each
coordinate direction or symmetry with respect to exchanging coordinate
axes). We note that dimensions $n>3$ can be treated by assuming higher
regularity properties.

\begin{theorem}[Approximation]\label{thm:main}
  Let $\eps = \eps_l \to 0$ be a sequence of positive numbers and $n
  \in \{1, 2, 3\}$ be the dimension. Let the medium $a_Y:\R^n\to
  \R^{n\times n}$ satisfy Assumption \ref {ass:a-f} and let the
  initial data $f: \R^n\to \R$ be as in \eqref {eq:f-cond}.  We use
  the coefficient matrices $A$ and $C$ defined in \eqref
  {eq:Taylor-mu}. Let $E$ and $F$ be positive semi-definite such that
  \eqref {eq:def-coeff-EF} holds (the existence is established in
  Lemma \ref {lem:decompose} below).  Then the following holds:
  \begin{enumerate}
  \item {\bf Well-posedness} Equation \eqref {eq:weakly-dispersive}
    with initial condition \eqref {eq:initial} has a unique solution
    $w^\eps$ for all positive times (see Theorem \ref
    {thm:main-approximation} below for function spaces).
  \item {\bf Error estimate} Let $w^\eps$ be the solution of \eqref
    {eq:weakly-dispersive}, and let $u^\eps$ be the solution of \eqref
    {eq:eps-wave} for the same initial condition \eqref
    {eq:initial}. Then, with a constant $C_0 = C_0(a_Y, T_0, f)$,
    there holds the error estimate
    \begin{equation}
      \label{eq:approx-result}
      \sup_{t\in [0,T_0 \eps^{-2}]} 
      \| u^\eps(.,t) - w^\eps(.,t) \|_{L^2(\R^n) + L^\infty(\R^n)} 
      \le C_0 \eps\,.
    \end{equation}
  \end{enumerate}
\end{theorem}

Here we use, for two Banach spaces with norms $\| . \|_X$ and $\|
. \|_Y$, the weaker norm $\| u \|_{X+Y} := \inf\{\|u_1\|_X + \|u_2\|_Y
: u = u_1 + u_2\}$. Such a norm appears in our main theorem \ref
{thm:main}, since different contributions to the error $u^\eps -
w^\eps$ are measured in different norms.

\subsection{Proof of Theorem \ref {thm:main}}
\label{ssec.errorestimate}

The following corollary is the central result of the Bloch analysis.
It is derived with mathematical rigor in \cite{Bloch-DLS-2013}; it
provides a comparison between the solution $u^\eps$ of the
heterogeneous wave equation with the explicitely defined function
$v^\eps$.

\begin{theorem}[Corollary 2.5 of
  \cite{Bloch-DLS-2013}] \label{thm:simplified-u} Let Assumption \ref
  {ass:a-f} be satisfied. Let $u^\eps$ be the solution of \eqref
  {eq:eps-wave} and let $v^\eps$ be defined by \eqref {eq:v-eps}. Then
  \begin{equation}
    \label{eq:approx-result-SanSym}
    \sup_{t\in [0,T_0 \eps^{-2}]} 
    \| u^\eps(.,t) - v^\eps(.,t) \|_{L^2(\R^n) + L^\infty(\R^n)} \le C_0 \eps.
  \end{equation}
\end{theorem}

The following theorem is based on energy methods. It provides the
comparison between the solution $w^\eps$ of the weakly dispersive
(homogeneous) equation and the explicit function $v^\eps$.

\begin{theorem}[Theorem 3.3 of
  \cite{Bloch-DLS-2013}]\label{thm:main-approximation}
  Let $A,C,E,F$ be tensors with the properties: $A\in \R^{n\times n}$
  is symmetric and positive definite, $\sum_{ij}A_{ij}\xi_i\xi_j\geq
  \gamma |\xi|^2$ for some $\gamma>0$, $E\in \R^{n\times n}$ and $F\in
  \R^{n\times n\times n\times n}$ are positive semi-definite and
  symmetric, $C\in \R^{n\times n\times n\times n}$ allows the
  decomposition \eqref {eq:def-coeff-EF}.  Then the following holds.

  \smallskip {\rm 1. Well-posedness.} For the initial datum $f\in
  H^2(\R^n)$, the equation
  \begin{equation}\label{eq:weakly-dispersive-copy}
    \begin{split}
      &\del_t^2 w^\eps - A D^2 w^\eps - \eps^2 \del_t^2 E D^2 w^\eps
      + \eps^2 F D^4 w^\eps = 0 \,,\\
      &w^\eps(.,0) = f, \qquad \del_t w^\eps(.,0) = 0\,
    \end{split}      
  \end{equation}
  has a unique solution $w^\eps \in L^\infty(0,T_0\eps^{-2} ;
  H^2(\R^n)) \cap W^{1,\infty}(0,T_0\eps^{-2} ; H^1(\R^n))$.\smallskip

  \smallskip {\rm 2. Approximation.} Let $v^\eps$ be defined by \eqref
  {eq:v-eps} where $F_0$ and $f$ are related by \eqref
  {eq:f-cond}. Let $w^\eps$ be a solution of
  \eqref{eq:weakly-dispersive-copy}. Then there holds
  \begin{equation}
    \label{eq:approx-result-v}
    \sup_{t\in [0,T_0 \eps^{-2}]}\|\del_t( v^\eps - w^\eps)(.,t) \|_{L^2(\R^n)}
    + \sup_{t\in [0,T_0 \eps^{-2}]}  \|\nabla( v^\eps - w^\eps)(.,t)\|_{L^2(\R^n)}
    \leq  C_0\eps^2
  \end{equation}
  with a constant $C_0$ that is independent of $\eps$.
\end{theorem}

Theorem \ref {thm:main} is a consequence of the above results and a
new decomposition lemma that is shown in the next subsection.  

The estimate \eqref {eq:approx-result-SanSym} provides that $\| u^\eps
- v^\eps \|$ is of order $\eps$. The norms coincide with the ones in
the claim \eqref {eq:approx-result}. It therefore remains to estimate
the difference $\| v^\eps - w^\eps \|$.

We define $A$ and $C$ through \eqref {eq:Taylor-mu} and note that $A$
is positive definite and symmetric.  The decomposition result of Lemma
\ref {lem:decompose} below allows to construct $E$ and $F$ such that
\eqref {eq:def-coeff-EF} is satisfied.  This means that Theorem \ref
{thm:main-approximation} can be applied.  It provides the
well-posedness claim and the estimate \eqref {eq:approx-result-v},
which shows that norms of derivatives of $v^\eps - w^\eps$ are of
order $\eps^2$.  The norms can be transformed with an interpolation
lemma (see Lemma 3.4 in \cite{Bloch-DLS-2013}); the result is an
estimate for $\sup_{t\in [0,T_0 \eps^{-2}]} \| v^\eps(.,t) -
w^\eps(.,t) \|_{L^2+L^\infty}$ of order $\eps$. We therefore obtain
\eqref {eq:approx-result}.

\subsection{Decomposition lemma}
\label{ssec.decomposition}

Our aim is to construct coefficient tensors $E \in \R^{n\times n}$ and
$F \in \R^{n\times n\times n\times n}$ such that the differential
operator $C D^4$ can be re-written as in \eqref {eq:def-coeff-EF}. We
impose that $E$ and $F$ are positive semi-definite and symmetric, i.e.
\begin{align}\label{eq:admissibledecompose}
  \sum_{i,j,k,l=1}^n F_{ijkl}\xi_{ij}\xi_{kl}\geq 0\quad \text{ for
    every }\xi\in\R^{n\times n}\,, \qquad F_{ijkl} = F_{klij}\,,
\end{align}
and similarly $\sum_{i,j=1}^n E_{ij}\eta_i\eta_j \geq 0$ for every
$\eta\in\R^n$ and $E_{ij}=E_{ji}$. The symmetry relations must hold
for all indices $i,j,k,l \in\{1,...,n\}$.  In this section, $n\in
\{1,2,3,4,\hdots\}$ can also be larger than $3$.

\begin{lemma}[Decomposability]\label{lem:decompose}
  Let $A\in \R^{n\times n}$ be a symmetric and positive definite
  matrix and let $C\in \R^{n\times n\times n\times n}$ be arbitrary.
  Then there exist symmetric and positive semi-definite tensors $E\in
  \R^{n\times n}$ and $F\in \R^{n\times n\times n\times n}$ such that
  the differential operator $CD^4$ can be written as in
  \eqref{eq:def-coeff-EF}.
\end{lemma}

The proof of the above lemma consists of two steps. In the first and
essential step we show that the decomposability result holds for
diagonal matrices $A\in\R^{n\times n}$ and with rotated derivative
operators.  In the second step, general matrices $A$ are treated by
diagonalization.

\begin{lemma}[Decomposability for diagonal matrices
  $A$]\label{lem:decomposediag}
  Let $A$ be a positive definite diagonal matrix, $A = \mathrm{diag}
  (a_1,a_2,...,a_n)$. Let $S\in SO(n)$ be an orthogonal matrix and let
  $C\in \R^{n\times n\times n\times n}$ be arbitrary.  Then there
  exist symmetric and positive semi-definite tensors $E\in \R^{n\times
    n}$ and $F\in \R^{n\times n\times n\times n}$ such that
  \begin{align}
    \label{eq:decomposerot}
    -C\tilde D^4 = E \tilde D^2 A \tilde D^2 - F  \tilde D^4\,.
  \end{align}
  Here, the operator $\tilde D := S D$ denotes a rotated derivative.
\end{lemma}

\begin{proof}[Proof of Lemma \ref{lem:decomposediag}] 

  {\em Step 1: Reduction to matrices $C$ with only one non-trivial
    entry.} We note that the relation \eqref {eq:decomposerot} is
  additive in the following sense: Let $C^{(1)}$ and $C^{(2)}$ be two
  matrices, and let \eqref {eq:decomposerot} be satisfied for
  $C^{(m)}$ with tensors $E^{(m)}$ and $F^{(m)}$, $m\in \{1,2\}$.
  Then \eqref {eq:decomposerot} holds for $C^{(1)} + C^{(2)}$ with the
  two tensors $E^{(1)} + E^{(2)}$ and $F^{(1)} + F^{(2)}$.  We exploit
  here that the sum of symmetric, semi-definite tensors is again
  symmetric and semi-definite.

  This observation implies that it is sufficient to consider a tensor
  $C$ that has only one non-trivial entry. We denote a tensor with
  only one entry $1$ in the canonical way as $e_\alpha\otimes
  e_\beta\otimes e_\gamma\otimes e_\delta$.  An arbitrary tensor $C$
  can be written as a sum, $C = \sum_{\alpha\beta\gamma\delta}
  C_{\alpha\beta\gamma\delta}\, e_\alpha\otimes e_\beta\otimes
  e_\gamma\otimes e_\delta$.  After constructing tensors
  $E^{(\alpha,\beta,\gamma,\delta)}$ and
  $F^{(\alpha,\beta,\gamma,\delta)}$ according to the one-entry tensor
  $C_{\alpha\beta\gamma\delta}\, e_\alpha\otimes e_\beta\otimes
  e_\gamma\otimes e_\delta$, we find $E$ and $F$ according to $C$ by a
  summation, $E = \sum_{\alpha\beta\gamma\delta}
  E^{(\alpha,\beta,\gamma,\delta)}$ and $F =
  \sum_{\alpha\beta\gamma\delta} F^{(\alpha,\beta,\gamma,\delta)}$.

  In the following construction, we restrict ourselves to a fixed
  choice of indices, $(\alpha,\beta,\gamma,\delta)\in
  \{1,...,n\}^4$. For a number $c\in \R$, we can consider a tensor $C$
  of the form $C = c\, e_\alpha\otimes e_\beta\otimes e_\gamma\otimes
  e_\delta$.

  Our aim is to re-write the differential operator $- C \tilde D^4 =
  c\, \tilde D_\alpha \tilde D_\beta \tilde D_\gamma \tilde D_\delta$.
  We use here $\tilde D_i:=\sum_{j=1}^n S_{ij}\del_{x_j}$ for the
  $i$-th component of the rotated gradient. In the following, $\{a\}_+
  := \mathrm{max}\{0,a\}$ denotes the positive part of a number $a\in
  \R$.

  {\em Step 2: Construction of $E$ and $F$ for $C = c\,
    e_\alpha\otimes e_\beta\otimes e_\gamma\otimes e_\delta$, where at
    least two indices coincide.}

  {\bf Case 1.} The indices $\alpha,\beta,\gamma,\delta$ contain two
  different pairs, i.e.  $(\alpha,\beta,\gamma,\delta) = (i,i,j,j)$ or
  $(\alpha,\beta,\gamma,\delta) = (i,j,i,j)$ or
  $(\alpha,\beta,\gamma,\delta) = (i,j,j,i)$ for $i,j\in\{1,...,n\}$.
  We restrict ourselves to $(\alpha,\beta,\gamma,\delta) = (i,i,j,j)$,
  the permutations define the same operator $C\tilde D^4$. We define
  the tensors $E = E^{(\alpha,\beta,\gamma,\delta)}$ and $F =
  F^{(\alpha,\beta,\gamma,\delta)}$ through
  \begin{align*}
    E_{ii} &:=\frac{\{-c\}_+}{a_j}\,,
    \qquad F_{ijij} := \{c\}_+\,,
    \qquad F_{imim} := \frac{\{-c\}_+}{a_j}a_m\,,
  \end{align*}
  for all $m\in\{1,...,n\}$ with $m\neq j$. All other entries of $E$
  and $F$ are set to zero.

  {\it Properties of $E$ and $F$.} By definition $E$ and $F$ are
  symmetric and positive semi-definite. A direct calculation yields
  the decomposition property:
  \begin{align*}
    &E\tilde D^2A\tilde D^2-F\tilde D^4 \\
    &\qquad= \left(\frac{\{-c\}_+}{a_j}\tilde
      D_i^2\right)\left(\sum_{m} a_m \tilde D_m^2\right)
    -\{c\}_+\tilde D_i^2\tilde D_j^2 - \sum_{m\neq
      j}\frac{\{-c\}_+}{a_j}a_m\tilde D_i^2 \tilde D_m^2
    \displaybreak[2]\\
    &\qquad= \{-c\}_+\tilde D_i^2\tilde D_j^2 + \sum_{m\neq
      j}\frac{\{-c\}_+}{a_j}a_m\tilde D_i^2 \tilde D_m^2
    -\{c\}_+\tilde D_i^2\tilde D_j^2
    - \sum_{m\neq j}\frac{\{-c\}_+}{a_j}a_m\tilde D_i^2 \tilde D_m^2\\
    &\qquad= \left(\{-c\}_+-\{c\}_+\right)\tilde D_i^2\tilde D_j^2
    =-c\tilde D_i^2\tilde D_j^2 = - C \tilde D^4\,\,.
  \end{align*}

  {\bf Case 2.} The indices $\alpha,\beta,\gamma,\delta$ contain three
  identical entries, i.e.  $(\alpha,\beta,\gamma,\delta) = (i,i,i,j)$
  or $(\alpha,\beta,\gamma,\delta) = (i,i,j,i)$ or
  $(\alpha,\beta,\gamma,\delta) = (i,j,i,i)$ or
  $(\alpha,\beta,\gamma,\delta) = (j,i,i,i)$ for $i,j\in\{1,...,n\}$
  with $i\neq j$.  We restrict ourselves to
  $(\alpha,\beta,\gamma,\delta) = (i,i,i,j)$ since the other cases
  define the same differential operator.  We define the tensors $E =
  E^{(\alpha,\beta,\gamma,\delta)}$ and $F =
  F^{(\alpha,\beta,\gamma,\delta)}$ through
  \begin{align*}
    E_{ij} & := E_{ji} := -\frac{c}{2a_i} =: \tilde c\,,
    \quad E_{ii}:=E_{jj}:= |\tilde c|\\
    F_{imim}&:=F_{jmjm}:= \left|\tilde c\right| a_m\,
    \quad \text{ for all }m\in\{1,...,n\}\\
    F_{imjm}&:=F_{jmim}:= \tilde c a_m\,\quad\quad \text{ for all }
    m\in\{1,...,n\}\text{ with }m\neq i.
  \end{align*}
  All other entries of $E$ and $F$ are set to zero.

  {\it Properties of $E$ and $F$.} By definition, $E$ and $F$ are
  symmetric.  Concerning the positive semi-definiteness of $E$ and $F$
  we calculate, for arbitrary $\xi\in\R^n$ and $\zeta\in{R^{n\times
      n}}$,
  \begin{align*}
    &\sum_{l,m} E_{lm} \xi_{l}\xi_{m} = |\tilde c|\xi_{i}^2 + |\tilde
    c|\xi_j^2 + 2\tilde c\xi_i\xi_j \geq |\tilde c|\xi_{i}^2 + |\tilde
    c|\xi_j^2
    - \left(|\tilde c|\xi_i^2 + |\tilde c|\xi_j^2\right)=0,\\
    &\sum_{l,m,p,q} F_{lmpq} \zeta_{lm} \zeta_{pq}
    = \sum_{m} |\tilde c|a_m\zeta_{im}^2 + \sum_{m}|\tilde
    c|a_m\zeta_{jm}^2 + \sum_{m\neq i} 2\tilde ca_m\zeta_{im}\zeta_{jm}\\
    &\qquad \geq \sum_{m}|\tilde c|a_m\zeta_{im}^2 + \sum_{m}|\tilde
    c|a_m\zeta_{jm}^2 - \sum_{m\neq i}(|\tilde
    c|a_m\zeta_{im}^2+|\tilde c|a_m\zeta_{jm}^2)\geq 0\,.
  \end{align*}
  Concerning the decomposition property we calculate
  \begin{align*}
    &E\tilde D^2A\tilde D^2-F\tilde D^4\\
    &\quad =\left(|\tilde c|\tilde D_i^2 + |\tilde c|\tilde D_j^2
      -\frac{c}{a_i}\tilde D_i\tilde D_j\right)\left(\sum_{m} a_m \tilde D_m^2\right)\\
    &\qquad -\sum_m |\tilde c|a_m (\tilde D_i^2\tilde D_m^2 +\tilde
    D_j^2\tilde D_m^2) + \sum_{m\neq i}a_m\frac{c}{a_i}\tilde
    D_i\tilde D_j\tilde D_m^2
    \displaybreak[2]\\
    &\quad =\sum_m |\tilde c|a_m (\tilde D_i^2\tilde D_m^2 +\tilde
    D_j^2\tilde D_m^2) - c\tilde D_i^3\tilde D_j
    - \sum_{m\neq i}a_m\frac{c}{a_i}\tilde D_i\tilde D_j\tilde D_m^2\displaybreak[2]\\
    &\qquad -\sum_m |\tilde c|a_m (\tilde D_i^2\tilde D_m^2 +\tilde
    D_j^2\tilde D_m^2) + \sum_{m\neq i}a_m\frac{c}{a_i}\tilde
    D_i\tilde D_j\tilde D_m^2 = -c\tilde D_i^3\tilde D_j = - C \tilde
    D^4\,\,.
  \end{align*}

  {\bf Case 3.} The indices $\alpha,\beta,\gamma,\delta$ contain two
  identical entries, the other entries are different. We restrict
  ourselves to the case $(\alpha,\beta,\gamma,\delta) = (i,i,j,k)$
  with three different indices $i, j, k \in \{1,...,n\}$, permutations
  of the indices define the same operator.  We define the tensors $E =
  E^{(\alpha,\beta,\gamma,\delta)}$ and $F =
  F^{(\alpha,\beta,\gamma,\delta)}$ through
  \begin{align*}
    E_{jk}&:=E_{kj}:=-\frac{c}{2a_i}=:\tilde c,\quad E_{kk}:=E_{jj}:=|\tilde c|\\
    F_{kmkm}&:=F_{jmjm}:= |\tilde c|a_m,\quad \text{ for all }m\in\{1,...,n\}\\
    F_{kmjm}&:=F_{jmkm}:=\tilde ca_m\quad\quad 
    \text{ for all }m\in\{1,...,n\}\text{ with }m\neq i.
  \end{align*}
  All other entries of $E$ and $F$ are set to zero.

  {\it Properties of $E$ and $F$.} By definition, $E$ and $F$ are
  symmetric.  Concerning the positive semi-definiteness of $E$ and
  $F$, we calculate for arbitrary $\xi\in\R^n$ and
  $\zeta\in{R^{n\times n}}$
  \begin{align*}
    &\sum_{l,m} E_{lm} \xi_{l} \xi_{m} = |\tilde c|\xi_{k}^2 + |\tilde
    c|\xi_j^2 + 2\tilde c\xi_k\xi_j \ge 0\,,\\
    &\sum_{l,m,p,q} F_{lmpq} \zeta_{lm} \zeta_{pq} = \sum_{m}
    |\tilde c| a_m \zeta_{km}^2 + \sum_{m}|\tilde c|a_m\zeta_{jm}^2
    + \sum_{m\neq i} 2\tilde ca_m\zeta_{km}\zeta_{jm} \ge 0\,.
  \end{align*}
  Regarding the decomposition property we calculate
  \begin{align*}
    &E\tilde D^2A\tilde D^2-F\tilde D^4\\
    &\quad = \left(|\tilde c|\tilde D_k^2 + |\tilde c|\tilde D_j^2
      -\frac{c}{a_i}\tilde D_k\tilde D_j\right)\left(\sum_{m} a_m \tilde D_m^2\right)\\
    &\qquad -\sum_m |\tilde c|a_m (\tilde D_k^2\tilde D_m^2 +\tilde
    D_j^2\tilde D_m^2) +
    \sum_{m\neq i}a_m\frac{c}{a_i}\tilde D_k\tilde D_j\tilde D_m^2\\
    &\quad =\sum_m |\tilde c|a_m (\tilde D_k^2\tilde D_m^2 +\tilde
    D_j^2\tilde D_m^2) - c\tilde D_k\tilde D_j\tilde D_i^2
    - \sum_{m\neq i}a_m\frac{c}{a_i}\tilde D_k\tilde D_j\tilde D_m^2\\
    &\qquad -\sum_m |\tilde c|a_m (\tilde D_k^2\tilde D_m^2 +\tilde
    D_j^2\tilde D_m^2) + \sum_{m\neq i}a_m\frac{c}{a_i}\tilde
    D_k\tilde D_j\tilde D_m^2 = -c\tilde D_k\tilde D_j \tilde D_i^2 =
    - C \tilde D^4\,\,.
  \end{align*}

  {\em Step 3: Construction of $E$ and $F$ for $C = c\,
    e_\alpha\otimes e_\beta\otimes e_\gamma\otimes e_\delta$, where no
    two indices coincide.}  We treat now the case
  $(\alpha,\beta,\gamma,\delta) = (i,j,k,l)$ with pairwise different
  indices $i,j,k,l\in\{1,...,n\}$; our aim is to rewrite the operator
  $C \tilde D^4 = c \tilde D_i\tilde D_j\tilde D_k\tilde D_l$. We
  obtain corresponding matrices in two steps: (a) We define a positive
  semi-definite, symmetric tensor $\hat F\in\R^{n\times n\times
    n\times n}$ in such a way that $\bar C := C - \hat F$ has only
  non-trivial entries at positions with repeated indices.  (b) We
  apply Step 2 of this proof to the remainder $\bar C$, which provides
  $\bar E$ and $\bar F$ with $-\bar C\tilde D^4 = \bar E \tilde D^2 A
  \tilde D^2 - \bar F \tilde D^4$. The desired tensors $E =
  E^{(\alpha,\beta,\gamma,\delta)}$ and $F =
  F^{(\alpha,\beta,\gamma,\delta)}$ are then obtained as $E := \bar E$
  and $F := \bar F + \hat F$.

  \smallskip We set
  \begin{equation}\label{eq:no-two-coincide}
    \hat F_{ijkl} := \hat F_{klij} := \frac12 c\,, 
    \qquad \hat F_{ijij} := \hat F_{klkl} := \frac12|c|\,.
  \end{equation}
  All other entries of $\hat F$ are set to zero. The symmetry of $\hat
  F$ is obvious; regarding positivity we calculate
  \begin{align*}
    \sum_{m,p,q,r}\hat F_{mpqr} \zeta_{mp}\zeta_{qr} &=
    \frac12|c|\zeta_{ij}^2 + \frac12 |c|\zeta_{kl}^2 +
    c\zeta_{ij}\zeta_{kl} \ge 0\,.
  \end{align*}

  It remains to check that Step 2 of this proof can be applied to the
  remainder $\bar C := C - \hat F$. We evaluate
  \begin{align*}
    C D^4 - \hat F D^4 &= 
    c\,\tilde D_i\tilde D_j\tilde D_k\tilde D_l 
    - \left[ c\,\tilde D_i\tilde D_j\tilde D_k\tilde D_l
      + \frac12 |c| \tilde D_i^2\tilde D_j^2
      + \frac12 |c| \tilde D_k^2\tilde D_l^2\right]\\
    &= - \frac12 |c| \left(\tilde D_i^2\tilde D_j^2 +
      \tilde D_k^2\tilde D_l^2\right).
  \end{align*} 
  This operator has nontrivial entries only for repeated indices, it
  therefore posesses a decomposition by Step 2.  This concludes the
  proof of the lemma.
\end{proof}

With Lemma \ref{lem:decomposediag} at hand we are in the position to
prove the general decomposition result of Lemma \ref{lem:decompose}.

\begin{proof}[Proof of Lemma \ref{lem:decompose}] 
  The symmetry of $A$ implies that $A$ is diagonalizable: there exists
  an orthogonal matrix $S\in SO(n)$ such that
  \begin{equation*}
    A=S^T\tilde A S\quad\text{with }\ 
    \tilde A=\mathrm{diag}( a_1, a_2,..., a_n).
  \end{equation*}
  Since $A$ is positive definite, the eigenvalues $a_i$ are positive.
  Our aim is to apply Lemma \ref{lem:decomposediag} with the diagonal
  matrix $\tilde A$ and a tensor $\tilde C$, which is defined from $C$
  with the transformation $S$.  Lemma \ref{lem:decomposediag} provides
  tensors $\tilde E$ and $\tilde F$ that can be transformed back into
  the desired tensors $E$ and $F$.

  {\em Step 1: Construction of $\tilde C$.}  Here, we denote the space
  of matrices by $M := \R^{n\times n}$. The tensor $C$ defines a
  linear map $C: M\to M$ through $C (e_i\otimes e_j))_{kl} =
  C_{ijkl}$. We define a transformed tensor $\tilde C: M \to M$
  through
  \begin{equation}\label{eq:CtildeC}
    \tilde C : M \ni B \mapsto S\cdot C(S^T B S)\cdot S^T\in M\,.
  \end{equation}

  As we show next, with $\tilde D:= SD$, the corresponding
  differential operators coincide, $CD^4= \tilde C\tilde D^4$. We use
  the convention that sums are over repeated indices.
  \begin{equation}\label{eq:calc8346}
    \begin{split}
    \tilde C\tilde D^4 &= \sum (\tilde C(e_i\otimes e_j))_{kl} \tilde
    D_i \tilde D_j \tilde D_k \tilde D_l \\
    &= \sum S_{k\alpha}\, C(S^T\cdot e_i\otimes e_j\cdot S)_{\alpha\beta}
    S_{l\beta} \tilde  D_i \tilde D_j \tilde D_k \tilde D_l\\
    &= \sum S_{k\alpha} S_{i\gamma} S_{j\delta}\, C(e_\gamma\otimes
    e_\delta)_{\alpha\beta} S_{l\beta}\
    S_{i\xi} D_\xi\, S_{j\zeta} D_\zeta\, S_{k\eta} D_\eta\, S_{l\theta} D_\theta\\
    &= \sum C(e_\gamma\otimes
    e_\delta)_{\alpha\beta} D_\alpha D_\beta D_\gamma D_\delta
    = CD^4\,.      
    \end{split}
  \end{equation}

  {\em Step 2: Application of Lemma \ref{lem:decomposediag}.}  Since
  $\tilde A$ is diagonal, we can apply Lemma \ref{lem:decomposediag}
  with the data $\tilde A$, $S$, and $\tilde C$. We find symmetric,
  positive semi-definite tensors $\tilde E\in \R^{n\times n}$ and
  $\tilde F\in\R^{n\times n\times n\times n}$ such that, with $\tilde
  D:= SD$,
  \begin{align}
    \label{eq:decomposerotated}
    -\tilde C\tilde D^4 = \tilde E \tilde D^2\tilde A\tilde D^2
    - \tilde F\tilde D^4.
  \end{align}

  We can now define the desired tensors $E$ and $F$ through $E:=S^T
  \tilde E S$ and
  \begin{equation}\label{eq:invert-rotation}
    F : M \ni B \mapsto S^T\cdot \tilde F(S B S^T)\cdot S\in M\,.
  \end{equation}
  Since $\tilde F$ is obtained from $F$ by the same formula as $\tilde
  C$ is obtained from $C$, cf. \eqref {eq:CtildeC}, the calculation of
  \eqref{eq:calc8346} yields $\tilde F\tilde D^4 = F D^4$.

  Regarding the operator $\tilde E \tilde D^2$ we calculate
  \begin{align*}
    \tilde E \tilde D^2 &= \sum (S E S^T)_{ij} \tilde D_i \tilde D_j
    = \sum S_{i\alpha} E_{\alpha\beta} S_{j\beta} S_{ik} D_k S_{jl} D_l\\
    &= \sum E_{kl} D_k D_l = E D^2\,.
  \end{align*}
  The calculation can also be applied to $A = S^T \tilde A S$ and
  provides $\tilde A \tilde D^2 = A D^2$.  We conclude that relation
  \eqref {eq:decomposerotated} coincides with
  \begin{align*}
    -CD^4 = ED^2AD^2 - FD^4\,.
  \end{align*}
  This is the decomposition result for the tensors $C$ and $A$. We
  remark that, since $S\in SO(n)$ is an orthogonal matrix, the
  symmetry and positive semi-definiteness of $\tilde E$ and $\tilde F$
  carry over to $E$ and $F$.  This concludes the proof of the general
  decomposition lemma.
\end{proof}

\section{An $\eps$-independent effective equation}
\label{sec.KdVrays}

The aim of this section is to carry the analysis of the weakly
dispersive effective equation one step further.  We will introduce
moving frame coordinates which will allow us to follow the main pulse
of $w^\eps$ along rays through the origin. Performing the limit
$\eps\rightarrow 0$ in distributional sense will provide an
$\eps$-independent linear third order equation (a linearized
KdV-equation) that describes the effective shape of the pulse in
dependence on the ray direction.

In the following we will restrict ourselves to the analysis of the
two-dimensional case with even symmetry,
\begin{equation}
  \label{eq:refl_sym}
  a_Y(y_1,y_2)=a_Y(-y_1,y_2)=a_Y(y_1,-y_2) \ \text{for all} \ y\in \R^2.
\end{equation}
The above symmetry assumption, which is in particular satisfied in the
case of a laminated structure, guarantees that the effective
coefficients $A$ and $C$ have the form (cf. proof of Lemma 2.6 in
\cite{Bloch-DLS-2013})
\begin{align}
  \label{eq:coefficients1}
  A&=\mathrm{diag}(a_1,a_2),\\
  \label{eq:coefficients2}
  C_{iiii}&=:\alpha_i,\quad C_{ijij}=C_{ijji}=C_{iijj}=:\beta\,\text{
    for }i,j\in\{1,2\}\text{ with }i\neq j.
\end{align}
All other entries of $C$, that are not mentioned above, vanish.

We expect that the main pulse of $w^\eps$, solution to the weakly
dispersive equation, propagates with a direction-dependent speed
according to the anisotropic matrix $A=\mathrm{diag}(a_1,a_2)$. We
introduce appropriate elliptic coordinates $(r,\varphi)$ through
\begin{align}
  \label{eq:ellipticcoord}
  (x_1,x_2) = (r\sqrt{a_1}\cos{\varphi}, r\sqrt{a_2}\sin{\varphi}).
\end{align}
The above coordinate transform is chosen in such a way that the main
pulse of $w^\eps$ at time t is located along the ellipse that is given
as the level set $\{ x\in \R^2 | r = t \}$. In order to perform a fine
analysis of the dynamics of the pulse, we rewrite $w^\eps$ as a
function of $(r,\varphi,t)$ and use a moving frame in the radial
variable $r$.  More precisely, given the solution $w^\eps$, we define
$W^\eps$ through
\begin{align}
  \label{eq:movinframefct}
  W^\eps(r,\varphi,t) :=
  \begin{cases}
    w^\eps\left(r+\frac{t}{\eps^2},\varphi,\frac{t}{\eps^2}\right)
    \qquad &\text{ for } r > -t \eps^{-2}\\
    0\qquad &\text{ for } r \le -t \eps^{-2}\,.
  \end{cases}
\end{align}
The time scaling $t/\eps^2$ accounts for the fact that the dispersive
effects of $w^\eps$ are weak, i.e. slow in time.  The main result of
this section is the following. Provided that $W^\eps$ has a
distributional limit $W$, then $W$ is characterized by an
$\eps$-independent linearized KdV-eqation.

\begin{proposition}[Effective behavior in the moving frame]
  \label{prop:movingframe} 
  Let the medium $a_Y:\R^2\rightarrow\R^{2\times 2}$ be evenly
  symmetric in the sense of \eqref{eq:refl_sym}.  Let
  $w^\eps(r,\varphi,t)$ be the solution to the weakly dispersive wave
  equation \eqref{eq:weakly-dispersive-copy}, expressed in elliptic
  coordinates. Let $W^\eps(r,\varphi,t)$ be defined by \eqref
  {eq:movinframefct}.  Assume that $W^\eps$ has a limit in the sense
  of distributions,
  \begin{align*}
    W^\eps\weakto W\quad\text{ in }\, \mathcal{D}'(\R \times\R\times
    (0,T)).
  \end{align*}
  Then the distribution $U:=\del_r W$ satisfies the following
  linearized cylindrical Korteweg-de-Vries-equation (linearized
  KdV-equation) in distributional sense
  \begin{align}
    \label{eq:effKdV}
    \del_t U + \frac{1}{2t}U -\frac{1}{2}\kappa(\varphi)\del^3_r U=0.
  \end{align}
  Here, the dispersion coefficient $\kappa$ is given by
  \begin{equation}  \label{eq:kappa-sincos}
    \kappa(\varphi) := 6\beta
    \frac{\cos^2{(\varphi)}}{a_1}\frac{\sin^2{(\varphi)}}{a_2} +
    \alpha_1\frac{\cos^4{(\varphi)}}{a_1^2} +
    \alpha_2\frac{\sin^4{(\varphi)}}{a_2^2}.
  \end{equation} 
\end{proposition}

The angle-dependent coefficient $\kappa(\varphi)$ can also be
expressed with $C$,
\begin{equation}
  \label{eq:kappa-C}
  \kappa(\varphi) = \sum_{ijkl} C_{ijkl} \xi_i \xi_j \xi_k \xi_l\quad
  \text{ for }\quad  
  \xi = \left( \frac1{\sqrt{a_1}} \cos(\fhi), \frac1{\sqrt{a_2}}
    \sin(\fhi)\right)\,.
\end{equation}
It can be understood as a measure for the amount of dispersion along
the ray $R_\varphi := \{ x = (r\sqrt{a_1}\cos{\varphi},
r\sqrt{a_2}\sin{\varphi}) \,|\, r\in (0,\infty) \}$.  Note that, due
to the negative semi-definiteness of $C$, the dispersion coefficient
is always nonpositive: $\kappa(\varphi)\leq 0$ for every angle
$\varphi$.

In the proof of the above proposition we will need some elementary
formulas, collected in the following remark.

\begin{remark}[Derivatives in elliptic coordinates]
  \label{rem:derivativeelliptic}
  Let $(r,\varphi)$ denote the elliptic coordinates defined in
  \eqref{eq:ellipticcoord}. Then the following relationship between
  derivatives in Cartesian and elliptic coordinates holds:
  \begin{align}
    \label{eq:der1}
    \del_{x_1} &=\frac{1}{\sqrt{a_1}}\left(\cos{\varphi}\,\del_r
      -\frac{1}{r}\sin{\varphi}\,\del_\varphi\right),\\
    \label{eq:der2}
    \del_{x_2} &=\frac{1}{\sqrt{a_2}}\left(\sin{\varphi}\,\del_r
      +\frac{1}{r}\cos{\varphi}\,\del_\varphi\right),\\
    \label{eq:der3}
    AD^2&=a_1\del^2_{x_1}+a_2\del^2_{x_2}
    =\del^2_r+\frac{1}{r}\del_r+\frac{1}{r^2}\del^2_{\varphi}\,.
  \end{align}
  For general derivatives of order $\gamma = \gamma_1+\gamma_2$ with
  $\gamma_1,\gamma_2\in\N_0$ one has
  \begin{align}
    \begin{split}
      \label{eq:multiderivative}
      \del^{\gamma_1}_{x_1}\del^{\gamma_2}_{x_2}&=
      \frac{1}{(\sqrt{a_1})^{\gamma_1}(\sqrt{a_2})^{\gamma_2}}P^{\gamma_1,\gamma_2}
      \left(\cos{\varphi},\sin{\varphi},\del_r,\del_\varphi,\tfrac{1}{r}\right)\\
      &=\frac{1}{(\sqrt{a_1})^{\gamma_1}(\sqrt{a_2})^{\gamma_2}}
      \left(\cos^{\gamma_1}{(\varphi)}\sin^{\gamma_2}{(\varphi)}\del^{\gamma}_r
        + \tilde P^{\gamma_1,\gamma_2}(\cos{\varphi},\sin{\varphi},\del_r,\del_\varphi,\tfrac{1}{r})\right),
    \end{split}
  \end{align}
  where $P^{\gamma_1,\gamma_2}$ and $\tilde P^{\gamma_1,\gamma_2}$ are
  polynomials and $\tilde P^{\gamma_1,\gamma_2}$ can be written as
  \begin{align*}
    \tilde
    P^{\gamma_1,\gamma_2}(\cos{\varphi},\sin{\varphi},\del_r,\del_\varphi,\tfrac{1}{r})
    = \sum_{k=1}^{\gamma}
    Q^{\gamma_1,\gamma_2}_k(\cos{\varphi},\sin{\varphi},\del_r,\del_\varphi)\frac{1}{r^k}\,,
  \end{align*}
  where $Q^{\gamma_1,\gamma_2}_k$ is a polynomial in $\cos{\varphi},
  \sin{\varphi}, \del_r, \del_\varphi$ of degree at most $2\gamma$.
\end{remark}

Equations \eqref {eq:der1}--\eqref {eq:der3} are elementary identities
for elliptic coordinates; \eqref {eq:multiderivative} follows easily
by an induction argument.  In the decomposition of
$P^{\gamma_1,\gamma_2}$, we wrote the term of highest order in
$\del_r$ explicitely, with the result that all other terms (collected
in $\tilde P^{\gamma_1,\gamma_2}$) contain non-vanishing powers of
$\frac{1}{r}$.

A consequence of the formulas \eqref {eq:der1}-\eqref
{eq:multiderivative} is the following.

\begin{lemma}
  \label{lem:auxiliarydistr}
  Let $w^\eps$ and $W^\eps$ be related by moving frame coordinates as
  in \eqref {eq:movinframefct} and let $W$ be a distributional limit
  of $W^\eps$ as in Proposition \ref{prop:movingframe}. Let
  $\gamma_1,\gamma_2\in\N_0$ and $\gamma = \gamma_1+\gamma_2$.  Then
  the following distributional convergence holds
  \begin{align*}
    \left(\del^{\gamma_1}_{x_1}\del^{\gamma_2}_{x_2} w^\eps\right)
    \left(r+\frac{t}{\eps^2},\varphi, \frac{t}{\eps^2}\right) \weakto
    \frac{\cos^{\gamma_1}{(\varphi)}
      \sin^{\gamma_2}{(\varphi)}}{(\sqrt{a_1})^{\gamma_1}(\sqrt{a_2})^{\gamma_2}}
    \,\del^{\gamma}_r\, W\quad
  \end{align*}
  for $\eps\rightarrow 0$ in $\mathcal{D}'(\R\times\R\times (0,T))$.
\end{lemma}

\begin{proof}
  We recall the definition $W^\eps(r,\varphi,t) =
  w^\eps\left(r+\tfrac{t}{\eps^2}, \varphi,\tfrac{t}{\eps^2}\right)$.
  The formula for $\del^{\gamma_1}_{x_1}\del^{\gamma_2}_{x_2}$ of
  Remark \ref{rem:derivativeelliptic} provides
  \begin{align}
    \begin{split}
      \label{eq:multiderivative2}
      &\left(\del^{\gamma_1}_{x_1}\del^{\gamma_2}_{x_2} w^\eps\right)
      \left(r+\tfrac{t}{\eps^2},\varphi,\tfrac{t}{\eps^2}\right)=\\
      &\frac{1}{(\sqrt{a_1})^{\gamma_1}(\sqrt{a_2})^{\gamma_2}}
      \left[\cos^{\gamma_1}{(\varphi)}\sin^{\gamma_2}{(\varphi)}\left(\del^{\gamma}_r
          w^\eps\right)\left(r+\tfrac{t}{\eps^2},\varphi,\tfrac{t}{\eps^2}\right)
        + R^\eps(r,\varphi,t)\right],
    \end{split}
  \end{align}
  where $R^\eps(r,\varphi,t)$ is of the form
  \begin{align*}
    R^\eps(r,\varphi,t)&:= \sum_{k=1}^{\gamma}\left(
      Q^{\gamma_1,\gamma_2}_k(\cos{\varphi},\sin{\varphi},\del_r,\del_\varphi)w^\eps\right)
    \left(r+\tfrac{t}{\eps^2},\varphi,\tfrac{t}{\eps^2}\right)\frac{1}{(r+\tfrac{t}{\eps^2})^k}\\
    &=\sum_{k=1}^{\gamma}
    Q^{\gamma_1,\gamma_2}_k(\cos{\varphi},\sin{\varphi},\del_r,\del_\varphi)W^\eps(r,\varphi,t)
    \left(\frac{\eps^2}{\eps^2r+t}\right)^k.
  \end{align*}
  The last equality holds, since the derivatives $\del_r,\del_\varphi$
  act on $w^\eps$ as they act on $W^\eps$.

  Since $W^\eps$ converges to $W$ in the sense of distributions, also
  all derivatives converge in the distributional sense.  We conclude
  $Q^{\gamma_1,\gamma_2}_k(\cos{\varphi},\sin{\varphi},\del_r,\del_\varphi)
  W^\eps(r,\varphi,t)$ $\weakto
  Q^{\gamma_1,\gamma_2}_k(\cos{\varphi},\sin{\varphi},\del_r,\del_\varphi)W$
  in the distributional sense on $\R\times\R\times(0,T)$ for every
  $k$. In particular, exploiting that $\tfrac{\eps^2}{\eps^2r+t}$ is
  of order $\eps^2$ for every $t>0$, we obtain $R^\eps\weakto 0$ in
  $\mathcal{D}'(\R\times\R\times(0,T))$.

  The same argument implies for the term containing only
  $r$-derivatives
  \begin{align*}
    \left(\del^{\gamma}_r
      w^\eps\right)\left(r+\tfrac{t}{\eps^2},\varphi,\tfrac{t}{\eps^2}\right)
    =\left(\del^{\gamma}_r
      W^\eps\right)\left(r,\varphi,t\right)\weakto\del^{\gamma}_r W
    \quad \text{ in }\ \mathcal{D}'(\R\times\R\times(0,T)).
  \end{align*}
  This allows to pass to the distributional limit in
  \eqref{eq:multiderivative2}, which concludes the proof of the lemma.
\end{proof}

\begin{proof}[Proof of Proposition \ref{prop:movingframe}]  
  We start from the weakly dispersive equation
  \begin{align*}
    \del_t^2 w^\eps - A D^2 w^\eps - \eps^2 \del_t^2 E
    D^2 w^\eps + \eps^2 F D^4 w^\eps = 0.
  \end{align*}
  The general idea of the proof is to transform the above equation
  into the elliptic coordinates of \eqref{eq:ellipticcoord}, to
  rewrite it in terms of $W^\eps(r,\varphi,t) =
  w^\eps\left(r+\tfrac{t}{\eps^2}, \varphi,\tfrac{t}{\eps^2}\right)$,
  and to pass to the distributional limit.


  {\em Step 1: The term $(\del_t^2 w^\eps - A D^2 w^\eps)$:} We start
  with the evaluation of time-derivatives of $w^\eps$, using the chain
  rule on $w^\eps(r,\varphi,t) = W^\eps(r-t, \varphi, t\eps^2)$,
  \begin{align}
    \label{eq:timeder2}
    (\del^2_t w^\eps)(r+\tfrac{t}{\eps^2}, \varphi, \tfrac{t}{\eps^2})
    &= \del^2_r W^\eps(r,\varphi,t) - 2\eps^2\del_t\del_r
    W^\eps(r,\varphi,t) + \eps^4\del^2_t W^\eps(r,\varphi,t).
  \end{align}
  Combining this result with \eqref{eq:der3}, we find
  \begin{align*}
    &(\del_t^2 w^\eps - A D^2 w^\eps)\left(r+\tfrac{t}{\eps^2},
      \varphi, \tfrac{t}{\eps^2}\right)=\\
    &\quad = \del^2_r W^\eps(r,\varphi,t) - 2\eps^2\del_t\del_r
    W^\eps(r,\varphi,t)
    + \eps^4\del^2_t W^\eps(r,\varphi,t)\\
    &\qquad -\left(\del^2_r
      W^\eps(r,\varphi,t)+\frac{1}{r+\tfrac{t}{\eps^2}}\del_r
      W^\eps(r,\varphi,t)
      +\frac{1}{(r+\tfrac{t}{\eps^2})^2}\del^2_{\varphi}W^\eps(r,\varphi,t)\right)\\
    &\quad = -\eps^2\left(2\del_t\del_r  W^\eps(r,\varphi,t)+\frac{1}{\eps^2r+t}\del_r W^\eps(r,\varphi,t)\right)\\
    &\qquad +\eps^4\left(\del^2_t W^\eps(r,\varphi,t) -
      \frac{1}{(\eps^2r+t)^2}\del^2_{\varphi}W^\eps(r,\varphi,t)\right).
  \end{align*}
  We divide by $\eps^2$ and take the distributional limit, exploiting
  that, by assumption, $W^\eps\weakto W$ in
  $\mathcal{D}'(\R\times\R\times(0,T))$. We obtain
  \begin{align}
    \label{eq:convscaled}
    \frac{1}{\eps^2}(\del_t^2 w^\eps - A D^2
    w^\eps)\left(r+\tfrac{t}{\eps^2}, \varphi,
      \tfrac{t}{\eps^2}\right) \weakto -2\del_t\del_r
    W-\frac{1}{t}\del_r W
  \end{align}
  for $\eps\rightarrow 0$ in $\mathcal{D}'(\R\times\R\times(0,T))$. 

  \smallskip {\em Step 2: The term $-\eps^2 \del_t^2 E D^2 w^\eps +
    \eps^2 F D^4 w^\eps$:} In view of the scaling in
  \eqref{eq:convscaled}, we have to analyze the distributional limit
  of $-E D^2\del_t^2 w^\eps + F D^4 w^\eps$.

  We recall that, by our construction of $E$ and $F$, the term can be
  written as $CD^4 w^\eps + R^\eps$ with some remainder $R^\eps$ that
  vanishes in the distributional limit as $\eps\rightarrow 0$. Indeed,
  exploiting the solution property of $w^\eps$ and the decomposition
  \eqref{eq:def-coeff-EF} one obtains
  \begin{align}
    -E D^2 \del_t^2 w^\eps +F D^4 w^\eps=CD^4 w^\eps 
    + \eps^2\left(-ED^2ED^2\del^2_t w^\eps + ED^2FD^4 w^\eps\right).
  \end{align}
  In particular, inserting an argument in scaled variables, 
  \begin{align}
    \begin{split}
      \label{eq:decomposewshift}
      &-(E D^2 \del_t^2 w^\eps)\left(r+\tfrac{t}{\eps^2}, \varphi,
        \tfrac{t}{\eps^2}\right) +(F D^4
      w^\eps)\left(r+\tfrac{t}{\eps^2},
	\varphi, \tfrac{t}{\eps^2}\right)\\
      &\qquad = (CD^4 w^\eps)\left(r+\tfrac{t}{\eps^2}, \varphi,
        \tfrac{t}{\eps^2}\right) + R^\eps,
    \end{split}
  \end{align}
  where the remainder $R^\eps$ is given by
  \begin{align*}
    R^\eps &= \eps^2 \left[-(ED^2ED^2\del^2_t w^\eps)\left(r+\tfrac{t}{\eps^2}, \varphi,
        \tfrac{t}{\eps^2}\right) +(ED^2FD^4 w^\eps)\left(r+\tfrac{t}{\eps^2},
	\varphi, \tfrac{t}{\eps^2}\right)\right]\\
    &= -\eps^2\left(\del^2_r- 2\eps^2\del_t\del_r
      +\eps^4\del^2_t\right)\left((ED^2ED^2w^\eps)\left(r+\tfrac{t}{\eps^2},
	\varphi, \tfrac{t}{\eps^2}\right)\right)\\
    &\quad +\eps^2(ED^2FD^4 w^\eps)\left(r+\tfrac{t}{\eps^2}, \varphi,
      \tfrac{t}{\eps^2}\right).
  \end{align*}
  In the second equality we used relation \eqref{eq:timeder2}. In this
  form we can apply Lemma \ref{lem:auxiliarydistr} to $R^\eps$. We
  obtain that $(ED^2ED^2w^\eps)\left(r+\tfrac{t}{\eps^2}, \varphi,
    \tfrac{t}{\eps^2}\right)$ and $(ED^2FD^4
  w^\eps)\left(r+\tfrac{t}{\eps^2}, \varphi, \tfrac{t}{\eps^2}\right)$
  have distributional limits. In particular, taking into account the
  $\eps^2$-factor in the above formula, we conclude the convergence
  $R^\eps\weakto 0$ in $\mathcal{D}'(\R\times\R\times(0,T))$ as
  $\eps\rightarrow 0$.
	
  Regarding the term $(CD^4 w^\eps)\left(r+\tfrac{t}{\eps^2},\varphi,
    \tfrac{t}{\eps^2}\right)$ in \eqref{eq:decomposewshift} one can
  directly apply Lemma \ref{lem:auxiliarydistr} with $\gamma=4$.
  Using the explicit form of $C$ from \eqref{eq:coefficients2}, we
  find
  \begin{align*}
    (CD^4 w^\eps)\left(r+\tfrac{t}{\eps^2},\varphi,
      \tfrac{t}{\eps^2}\right)\weakto \left(6\beta
      \frac{\cos^2{(\varphi)}}{a_1}\frac{\sin^2{(\varphi)}}{a_2} +
      \alpha_1\frac{\cos^4{(\varphi)}}{a_1^2} +
      \alpha_2\frac{\sin^4{(\varphi)}}{a_2^2}\right) \del_r^4 W
  \end{align*}
  in $\mathcal{D}'(\R\times\R\times(0,T))$, as $\eps\rightarrow 0$. 
  
  \smallskip {\em Step 3: Conclusion.} Summing up the various terms, we
  find that the distribution $W$ satisfies the equation
  \begin{align*}
    -2\del_t\del_r W-\frac{1}{t}\del_r W + \left(6\beta
      \frac{\cos^2{(\varphi)}}{a_1}\frac{\sin^2{(\varphi)}}{a_2} +
      \alpha_1\frac{\cos^4{(\varphi)}}{a_1^2} +
      \alpha_2\frac{\sin^4{(\varphi)}}{a_2^2}\right) \del_r^4 W=0\,.
  \end{align*}
  For $U:=\del_r W$ we obtain the equation
  \begin{align*}
    \del_t U+\frac{1}{2t}U-\frac{1}{2}\left(6\beta
      \frac{\cos^2{(\varphi)}}{a_1}\frac{\sin^2{(\varphi)}}{a_2} +
      \alpha_1\frac{\cos^4{(\varphi)}}{a_1^2} +
      \alpha_2\frac{\sin^4{(\varphi)}}{a_2^2}\right) \del_r^3 U=0\,.
  \end{align*}
  This concludes the proof of Proposition \ref{prop:movingframe}. 
\end{proof}

\section{Calculation of approximate solutions}
\label{sec.Calculation}

In this section we discuss practical aspects of determining the
coefficients in the effective dispersive equation \eqref
{eq:weakly-dispersive}. We present a numerical method and compare
solutions $u^\eps$ of the original wave equation \eqref {eq:eps-wave}
with solutions $w^\eps$ of \eqref {eq:weakly-dispersive}.

\subsection{Cell problems}

The decomposition lemmas \ref{lem:decompose} and \ref
{lem:decomposediag} show that the coefficient tensors $A,E$ and $F$ of
the weakly dispersive effective equation can be calculated from $A$
and $C$, i.e. by the second and fourth derivatives of the Bloch
eigenvalue $\mu_0(k)$ at $k=0$. This fact makes the effective tensors
computable in terms of cell-problems. In the following calculation we
differentiate \eqref{eq:eigenvalue} with respect to $k$ and integrate
in $y$ to obtain formulas for $A$ and $C$. The result will be a
practical algorithm to determine $A$ and $C$. We will also transform
Lemmas \ref{lem:decompose} and \ref {lem:decomposediag} into an
algorithm that can be used to calculate $E$ and $F$.

For a wave parameter $k\in Z = (-1/2, 1/2)^n$ we consider the
Bloch eigenvalue problem \eqref {eq:eigenvalue} for $m=0$, i.e.
\begin{equation}
  \label{eq:eigenvalue-m=0}
  - (\nabla_y + \ri k ) \cdot (a_Y(y) (\nabla_y + \ri k ) \psi_0(y,k) ) 
  = \mu_0(k) \psi_0(y,k),
\end{equation}
where $\mu_0(k) \in \R$ is the smallest eigenvalue for each $k\in Z$.

\begin{remark} \label{rem:k0_analyt} For $k=0$, the eigenvalue is
  $\mu_0(0)=0$ and the eigenfunction $\psi_0(\cdot,0)$ is a constant
  function.  We normalize eigenfunctions so that
  \begin{equation}\label{E:normaliz}
    \langle \psi_0(\cdot,k) \rangle_Y
    := \frac{1}{|Y|} \int_Y \psi_0(y,k) \,dy  = 1\,,
  \end{equation}
  in particular we obtain $\psi_0(\cdot,0) \equiv 1$.  The
  normalization is possible in a neighborhood of $k=0$, since averages
  of the first eigenfunction $\psi_0(\cdot,k)$ do not vanish for small
  $|k|$. The eigenvalue map $k \mapsto \mu_0(k)\in \R$ and the
  eigenfunction map $k\mapsto \psi_0(\cdot,k)\in L^2(Y)$ are analytic,
  see \cite{MR1484944}.
\end{remark}

Due to Remark \ref{rem:k0_analyt}, it is legitimate to determine
derivatives of $\mu_0$ by differentiating the eigenvalue problem
\eqref{eq:eigenvalue-m=0} in $k$.  We use standard multi-index
notation with $\N_0 = \{0,1,2,...\}$: a multi-index $\alpha =
(\alpha_1,...,\alpha_n) \in \N_0^n$ has length $|\alpha| :=
\alpha_1+\alpha_2+...+\alpha_n$ and defines a differential operator
$\del^\alpha$ of order $|\alpha|$ (derivatives $\del_j = \del_{k_j}$
are with respect to $k\in \R^n$), $\del^\alpha := \del_1^{\alpha_1}
\del_2^{\alpha_2} ... \del_n^{\alpha_n}$.  We define, for
$\alpha\in\N_0^n$,
\begin{align*}
  \mu_0^\alpha &:= \del^\alpha\mu_0|_{k=0}\,,\qquad\qquad
  \psi_{0}^\alpha := \del^\alpha\psi_0|_{k=0}\,.
\end{align*}
Differentiating the normalization \eqref{E:normaliz}, we obtain that
the averages of the higher order derivatives vanish,
$$\langle \psi_{0}^\alpha \rangle_Y =0 \ \text{ for all } \alpha \in \N_0^n,\ 
\alpha \neq 0.$$ We additionally define the differential operators 
\begin{align*}
  \mathcal{A}(k) &:= -(\nabla_y+\ri k)\cdot\left(a_Y(y)(\nabla_y+\ri k)\right),\\
  \mathcal{A}^\alpha &:= \del^\alpha\mathcal{A}|_{k=0}\,.
\end{align*}

\begin{lemma}[The operators $\mathcal{A}^\alpha$]
  \label{lem:derivativeA}
  Let $e_j$ denote the $j$-th Euclidean unit vector. The operators
  $\mathcal{A}^\alpha$ with $|\alpha| \le 2$ are
  \begin{align*}
    \mathcal{A}^{0}f&=-\nabla_y\cdot(a_Y\nabla_y f),\\
    \mathcal{A}^{e_j}f
    &=-\ri\left[(a_Y\nabla_y f)\cdot e_j +
      \nabla_y\cdot(a_Y e_j f)\right]\\
    \mathcal{A}^{e_i+e_j}f&=2(a_Y)_{ij}f
  \end{align*}
  for $i,j=1,...,n$, where $f:\R^n\rightarrow\R$ is an arbitrary
  smooth function.  All operators $\mathcal{A}^\alpha$ with $|\alpha|
  \ge 3$ vanish identically.
\end{lemma}

\begin{proof}
  The formula for $\mathcal{A}^{0}$ is obtained from the definition of
  $\mathcal{A}(k)$ by setting $k=0$. Concerning the first order
  derivatives of $\mathcal{A}(k)$ we calculate
  \begin{align*}
    \del_j\mathcal{A}(k)f&=-\ri e_j\cdot(a_Y(\nabla_y+\ri k)f)
    -\ri(\nabla_y+\ri k)\cdot\left(a_Y e_j f \right)\\
    &=-\ri\left[e_j\cdot (a_Y(\nabla_y+\ri k)f)
      + (\nabla_y+\ri k)\cdot(a_Y e_j f)\right]\,.
  \end{align*}
  Inserting $k=0$ provides the claim about $\mathcal{A}^{e_j}$.  For
  the second order derivatives of $\mathcal{A}(k)$ we calculate
  \begin{align*}
    \del_i \del_j \mathcal{A}(k)f = -\ri^2[e_j \cdot (a_Y e_i f) 
    + e_i\cdot(a_Y e_j f)] = ((a_Y)_{ji}+(a_Y)_{ij})f =
    2(a_Y)_{ij}f,
  \end{align*}
  where the last equality holds due to the symmetry of $a_Y$.
\end{proof}

We next want to obtain equations that characterize the functions
$\psi_0^{\alpha}$. In order to calculate derivatives of products, we
will use the Leibniz formula in the following form: For every
multi-index $\alpha\in\N_0^n$ and sufficiently smooth functions $f, g
: \R^n\to\C$ there holds
\begin{align*}
  \del^\alpha(fg)=\sum_{\beta\in\N_0^n} {\alpha\choose\beta}\del^\beta
  f\ \del^{\alpha-\beta}g\,.
\end{align*}
We use the binomial coefficient ${\alpha\choose\beta} :=
{\alpha_1\choose\beta_1}{\alpha_2\choose\beta_2} \hdots
{\alpha_n\choose\beta_n}$, the (partial) ordering $\beta\le \alpha
:\Leftrightarrow \beta_i\leq\alpha_i$ for every $i\in\{1,...,n\}$, and
note that ${\alpha\choose\beta} \neq 0$ is non-vanishing only for
$\beta \le \alpha$. In particular, in the following calculations, all
sums over multi-indices are finite sums.  Summation of multi-indices
is performed in the standard way as $\alpha\pm\beta :=
(\alpha_1\pm\beta_1, \hdots ,\alpha_n\pm\beta_n)$.

With the operator $\mathcal{A}(k)$, we can write equation
\eqref{eq:eigenvalue-m=0} as $\mathcal{A}(k) \psi_0(y,k) = \mu_0(k)
\psi_0(y,k)$. Taking partial derivatives with respect to $k$ with the
Leibniz formula, we find the following result.

\begin{lemma}[Cell Problems for $\psi_0^\alpha$]
  \label{prop:cellprob}
  Let $\alpha\in\N_0^n$ be a multi-index. Then the function
  $\psi_0^{\alpha}$ satisfies the relation
  \begin{align*}
    \mathcal{A}^0\psi_0^\alpha
    +\sum_{j=1}^n\alpha_j\mathcal{A}^{e_j}\psi_0^{\alpha-e_j}
    +\sum_{i\le j=1}^n {\alpha\choose e_i + e_j} \mathcal{A}^{e_i+e_j}
    \psi_0^{\alpha-e_i-e_j} = \sum_{\beta\in\N_0^n}
    {\alpha\choose\beta} \mu_0^\beta\psi_0^{\alpha-\beta}\,.
  \end{align*}
  Inserting the operators from Lemma \ref {lem:derivativeA} and using
  $\mu_0^0 = \mu_0|_{k=0} = 0$, this equation reads
  \begin{align}
    \begin{split}
      \label{eq:cell1}
      -\nabla_y\cdot\left[a_Y\nabla_y\psi_0^\alpha\right]
      &=\ri\sum_{j=1}^n\alpha_j\left[(a_Y
        \nabla_y\psi_0^{\alpha-e_j})\cdot e_j
	+\nabla_y\cdot\left(\psi_0^{\alpha-e_j}a_Ye_j\right)\right]\\
      &\quad -2 \sum_{i\le j=1}^n {\alpha\choose e_i + e_j} (a_Y)_{ij}
      \psi_0^{\alpha-e_i-e_j} + \sum_{0\neq\beta\in \N_0^n}
      {\alpha\choose\beta} \mu_0^\beta\psi_0^{\alpha-\beta}.
    \end{split}
  \end{align}
\end{lemma}

In our next step we obtain formulas for $\mu_0^\alpha$ in terms of
$\psi_0^{\beta}$ and $\mu_0^{\beta}$ with $\beta<\alpha$.

\begin{lemma}[Formulas for $\mu_0^\alpha$]\label{lem:mu_der}
  Let $\alpha\in\N_0^n$ be a multi-index. In the case $\alpha = 0$ we
  have $\mu_0^\alpha = \mu_0^{0} = 0$. For $|\alpha|\geq 1$ there
  holds
  \begin{align}
    \mu_0^\alpha &= 0 \qquad \text{if } |\alpha| \text{ is odd,}\\
    \mu_0^\alpha &= 2 \sum_{i\le j=1}^n {\alpha\choose e_i + e_j}
    \left\la (a_Y)_{ij}\psi_0^{\alpha-e_i-e_j}\right\ra_Y \notag -
    \ri\sum_{j=1}^n\alpha_j\left\la a_Y
      \nabla_y\psi_0^{\alpha-e_j}\right\ra_Y\cdot e_j\label{eq:mu-der}\\
    & \hspace{8.5cm} \text{if } |\alpha| \text{ is even.}
  \end{align}
\end{lemma}

\begin{proof}
  The statement for $\mu_0^0$ has already been observed, cf. Remark
  \ref{rem:k0_analyt}. For odd $|\alpha|$, the symmetry
  $\mu_0(-k)=\mu(k)$ for all $k\in Z$ implies $\mu_0^\alpha=0$
  (compare e.g. Remark 2.7 in \cite{Bloch-DLS-2013}).  For even
  $|\alpha|$, the formula is a consequence of relation
  \eqref{eq:cell1}. Indeed, we can reorganize \eqref{eq:cell1},
  writing the term with $\beta = \alpha$ in the last sum explicitely,
  and find
  \begin{align*}
    \mu_0^\alpha\psi_0^0=
    &-\nabla_y\cdot\left[a_Y\nabla_y\psi_0^\alpha\right]
    -\ri\sum_{j=1}^n\alpha_j\left[(a_Y
      \nabla_y\psi_0^{\alpha-e_j})\cdot e_j
      +\nabla_y\cdot\left(\psi_0^{\alpha-e_j}a_Ye_j\right)\right]\\
    &+ 2 \sum_{i\le j=1}^n {\alpha\choose e_i + e_j}
    (a_Y)_{ij}\psi_0^{\alpha-e_i-e_j} - \sum_{1\leq |\beta|\leq
      |\alpha|-1} {\alpha\choose\beta}
    \mu_0^\beta\psi_0^{\alpha-\beta}.
  \end{align*}

  We integrate this relation over the periodicity cell $Y$, exploiting
  $\langle\psi_0^0\rangle_Y = 1$ and $\langle \psi_{0}^\alpha
  \rangle_Y = 0$ for all $\alpha \neq 0$. Furthermore, we use that the
  integral over the two terms in divergence form vanishes by
  periodicity, and obtain \eqref {eq:mu-der}.
\end{proof}

The above formulas allow to calculate all unknowns $\psi_0^{\alpha}$
and $\mu_0^\alpha$ in a recursive scheme. From Lemmas
\ref{prop:cellprob} and \ref{lem:mu_der}, we extract the following
algorithm for the computation of the tensors $A$ and $C$.

\begin{algorithm}(Computation of $A$ and $C$)\label{A:AC} The tensors
  $A$ and $C$ can be computed in five steps.
  \begin{enumerate}
  \item For $j \in \{1,\dots,n\}$, solve \eqref {eq:cell1} for $\alpha
    = e_j$, i.e.
    $$-\nabla_y\cdot [a_Y \nabla_y \psi_0^{e_j}] = \ri \nabla_y\cdot(a_Y e_j)\,.$$
  \item For $i,j \in \{1,\dots,n\}$, evaluate \eqref {eq:mu-der} for
    $\alpha = e_i + e_j$, i.e.
    $$A_{ij} = \frac{1}{2}\mu_0^{e_i+e_j} 
    = \left\la (a_Y)_{ij} - \frac{\ri}{2} \left((a_Y\nabla_y
        \psi_0^{e_j})\cdot e_i +(a_Y\nabla_y \psi_0^{e_i})\cdot e_j
      \right)\right\ra_Y\,.$$
  \item For $i,j \in \{1,\dots,n\}$, solve \eqref {eq:cell1} for
    $\alpha = e_i + e_j$, i.e.
    \begin{align*}
      -\nabla_y\cdot [a_Y \nabla \psi_0^{e_i+e_j}]
      =& \ri \left[(a_Y\nabla_y \psi_0^{e_j})\cdot e_i 
        +(a_Y\nabla_y \psi_0^{e_i})\cdot e_j \right. \\
      & \left. \quad + \nabla_y\cdot (\psi_0^{e_j}a_Ye_i)+ \nabla_y\cdot
        (\psi_0^{e_i}a_Ye_j)\right] -2(a_Y)_{ij} + \mu_0^{e_i+e_j}\,.
    \end{align*}
  \item For $i,j,k \in \{1,\dots,n\}$, solve \eqref {eq:cell1} for
    $\alpha = e_i + e_j + e_k$,
    \begin{align*}
      &    -\nabla_y\cdot [a_Y \nabla_y \psi_0^{e_i+e_j+e_k}] =\\
      &\quad = \ri \left[(a_Y \nabla_y\psi_0^{e_j+e_k})\cdot e_i
        +(a_Y \nabla_y\psi_0^{e_i+e_k})\cdot e_j +(a_Y \nabla_y\psi_0^{e_i+e_j})\cdot e_k \right.\\
      &\qquad\quad \left. + \nabla_y\cdot(\psi_0^{e_j+e_k}a_Ye_i)
        +\nabla_y\cdot(\psi_0^{e_i+e_k}a_Y e_j)
        +\nabla_y\cdot(\psi_0^{e_i+e_j}a_Ye_k)\right]\\
      &\qquad -2\left((a_Y)_{ij}\psi_0^{e_k}
        +(a_Y)_{ik}\psi_0^{e_j}+(a_Y)_{jk}\psi_0^{e_i}\right) \\
      &\qquad + \mu_0^{e_i+e_j}\psi_0^{e_k}+
      \mu_0^{e_i+e_k}\psi_0^{e_j}+ \mu_0^{e_j+e_k}\psi_0^{e_i}\,.
    \end{align*}
  \item For $i,j,k,l\in \{1,\dots,n\}$, evaluate \eqref {eq:mu-der}
    for $\alpha = e_i + e_j + e_k +e_l$ to find
    \begin{align*}
      C_{ijkl} &= \frac{1}{24}\mu_0^{e_i+e_j+e_k+e_l}\\
      &= \frac{1}{24} \Biggl\langle 2\left((a_Y)_{ij}\psi_0^{e_k+e_l}
        +(a_Y)_{ik}\psi_0^{e_j+e_l}+(a_Y)_{il}\psi_0^{e_j+e_k}\right.\\
      &\qquad\quad \left.+(a_Y)_{jk}\psi_0^{e_i+e_l}
        +(a_Y)_{jl}\psi_0^{e_i+e_k}+(a_Y)_{kl}\psi_0^{e_i+e_j}\right) \\
      &\qquad - \ri\left((a_Y\nabla_y\psi_0^{e_j+e_k+e_l})\cdot e_i
        +(a_Y\nabla_y\psi_0^{e_i+e_k+e_l})\cdot e_j \right.\\
      &\qquad\quad \left. +(a_Y\nabla_y\psi_0^{e_i+e_j+e_l})\cdot
        e_k+(a_Y\nabla_y\psi_0^{e_i+e_j+e_k})\cdot
        e_l\right)\Biggr\rangle_Y\,.
    \end{align*}
  \end{enumerate}
  The elliptic problems in Steps 1,3, and 4 are posed on $Y$ with
  periodic boundary conditions and with the condition of zero mean,
  i.e.
  $$\langle \psi_0^{e_j} \rangle_Y = \langle\psi_0^{e_i+e_j} \rangle_Y 
  =  \langle\psi_0^{e_i+e_j+e_k}\rangle_Y = 0.$$
\end{algorithm}

In the above algorithm we have used the fact that $\mu_0^{\alpha} = 0$
for all $\alpha$ with $|\alpha|=1$ and $|\alpha|=3$. The above
cell-problems are complex valued, but the resulting tensors $A$ and
$C$ are real.  The fact that values of $\psi_0^\alpha$ and
$\mu_0^\alpha$ do not change upon permutations of the entries
$\alpha_1, \alpha_2,\dots,\alpha_n$, allows to reduce the number of
problems to be solved in Steps 3 and 4: The relevant number is ${{n+1}
  \choose{2}}$ and ${{n+2} \choose{3}}$ rather than $n^2$ and $n^3$.
Moreover, spatial symmetries of $a_Y(.)$ can reduce the number of
problems further: If $a_Y(.)$ is even in both variables, $a_Y(y_1,y_2)
= a_Y(-y_1,y_2) = a_Y(y_1,-y_2)$ for all $y\in Y$, then all
derivatives of $\mu_0$ at $k=0$ involving an odd number of derivatives
in one direction vanish, and only $A_{ii}, C_{iiii}$, and $C_{iijj} =
C_{ijij} = C_{jiij}$ for $i,j\in \{1,\dots,n\}$ are potentially
nonzero.

\medskip At this point, we have presented a scheme to compute the
effective tensors $A$ and $C$ with the help of a sequence of
cell-problems. We conclude this section with the outline of an
algorithm (based on the proofs of Lemmas \ref{lem:decompose} and \ref
{lem:decomposediag}) that provides formulas for the effective
coefficient tensors $E$ and $F$.

\begin{algorithm}(Computation of $E$ and $F$)\label{A:EF} The tensors
  $E$ and $F$ can be computed in four steps.
  \begin{enumerate}
  \item Determine $S\in SO(n)$ such that $A = S^T\tilde A S$ with
    $\tilde A=\mathrm{diag}( a_1, a_2,..., a_n)$. Define $\tilde C$ by
    \eqref {eq:CtildeC}.
  \item Loop over all indices $1\le \alpha, \beta, \gamma, \delta\le
    n$ such that no two indices coincide (an empty set in dimension
    $n\le 3$). Use \eqref {eq:no-two-coincide} to compile $\hat F$ and
    set $\bar C := C - \hat F$.
  \item Loop over all remaining indices $1\le \alpha, \beta, \gamma,
    \delta\le n$. Corresponding to $\tilde A$ and $\bar C$, compile
    $\bar E$ and $\bar F$ from the explicit formulas of Cases 1-3 in
    the proof of Lemma \ref {lem:decomposediag}.
  \item Set $\tilde E := \bar E$ and $\tilde F = \bar F + \hat F$.
    Obtain $E := S^T \tilde E S$ and $F$ from \eqref
    {eq:invert-rotation}.
  \end{enumerate}
\end{algorithm}

\subsection{Numerical results in 2D}
\label{sec.numerics}

In the following, we present numerical results for all the three parts
of the homogenization problem: The computation of the original
$\eps$-problem, the computation of effective coefficients with the
help of cell-problems, and the computation of the weakly dispersive
effective problem. The methods vary, we use finite element schemes and
finite difference schemes, see e.g.\,\cite {Quarteroni2007}. We refer
also to \cite {MR3090137} for a recent analysis of numerical methods
and further references. Our results show an excellent agreement
between solutions $u^\eps$ of \eqref {eq:eps-wave} and solutions
$w^\eps$ of \eqref {eq:weakly-dispersive}.

The initial conditions for all the tests below are
\begin{equation}
\label{E:IC_tests}
u^\eps(x,0) = w^\eps(x,0) = e^{-4(x_1^2+x_2^2)}, 
\quad \partial_t u^\eps(x,0) = \partial_t w^\eps(x,0) =0.
\end{equation}
The periodicity cell is $Y=[-\pi,\pi]^2$ in the $y-$variables and
$\eps Y$ in the $x-$variables. All the tests are carried out for the
case of even symmetry in $a_Y$,
i.e. $a_Y(-y_1,y_2)=a_Y(y_1,-y_2)=a_Y(y_1,y_2)$ for all $y\in Y$, so
that for the even initial data in \eqref{E:IC_tests} problems
\eqref{eq:eps-wave} and \eqref{eq:weakly-dispersive} can be reduced to
one quadrant with homogeneous Neumann boundary conditions along the
coordinate axes $x_1=0, x_2=0$. The computational domain $\Omega$ is
rectangular with two sides coinciding with the negative $x_1$ and
$x_2$ axes. At the remaining two sides of the rectangle we use
homogeneous Dirichlet boundary conditions. The size of $\Omega$ is
chosen so that the solution remains localized in $\Omega$ until the
final computation time. The function $a_Y$ is chosen piecewise
constant in all tests; the observed error convergence agrees with
\eqref {eq:approx-result} although Theorem \ref {thm:main} treats only
differentiable fields $a_Y$.

The even symmetry of $a_Y$ in both $y_1$ and $y_2$ causes that $A$ is
diagonal and $C$ has only eight nonzero entries, see \eqref
{eq:coefficients1}--\eqref {eq:coefficients2}:
\begin{align*}
&A_{ii} =: a_i\,, \qquad C_{iiii} =: \alpha_i\,, \qquad i\in\{1,2\}\,,\\
&C_{1122} = C_{2211} = C_{1212} = C_{2121} = C_{1221} = C_{2112} =: \beta\,,
\end{align*}
with all other entries of $A$ and $C$ zero.  A choice of $E$ and $F$
according to Algorithm \ref {A:EF} is
\begin{align}
  \begin{split}\label{E:EF_sym}
    E_{11} &= \frac{\{-\alpha_1\}_+}{a_1}+3\frac{\{-\beta\}_+}{a_2}, \ 
    E_{22} = \frac{\{-\alpha_2\}_+}{a_2}+3\frac{\{-\beta\}_+}{a_1}, \ E_{12} = E_{21} =0,\\
    F_{1111} &= \{\alpha_1\}_+ + 3\frac{a_1}{a_2}\{-\beta\}_+, \ 
    F_{2222} = \{\alpha_2\}_+ + 3\frac{a_2}{a_1}\{-\beta\}_+,\\
    F_{2121} &= \frac{a_1}{a_2}\{-\alpha_2\}_+ + 3\{\beta\}_+, \ 
    F_{1212} = \frac{a_2}{a_1}\{-\alpha_1\}_+ + 3\{\beta\}_+\\
  \end{split}
\end{align}
with all other entries of $F$ being zero. We use these tensors in the
numerical tests below. Note that $\alpha_i\le 0$ holds for $i\in
\{1,2\}$, hence $\{-\alpha_i\}_+ = -\alpha_i$ and $\{\alpha_i\}_+ =
0$.

\subsubsection*{Numerical method}

The values $a_j, \alpha_j$, and $\beta$ for $j=1,2$ are computed via
Algorithm \ref{A:AC}, where the elliptic equations in Steps 1, 3, and
4 are discretized by linear finite elements using the PDE-Toolbox of
Matlab. The periodic boundary conditions are implemented by modifying
the stiffness matrix and the load vector corresponding to homogeneous
Neumann boundary conditions. In all tests a uniform discretization
conforming to the material geometry is generated with the Matlab
function \texttt{poimesh}, i.e. the elements are all right angled,
have equal size and the discontinuity lines of $a_Y$ intersect no
elements. The element size is given by specifying the spacings $h_1$
and $h_2$, i.e. the lengths of the triangle legs in \texttt{poimesh}.

The original wave equation \eqref{eq:eps-wave} is discretized in space
also via linear finite elements using the PDE-Toolbox of Matlab with
uniform elements (generated by \texttt{poimesh}) conforming to the
geometry. The values of $h_1$ and $h_2$ are given for each example
below.

For the weakly dispersive problem \eqref{eq:weakly-dispersive}, which
has constant coefficients, we use the fourth order centered finite
difference discretization in space for both the second order and
fourth order derivatives. In all tests we use the spacing $dx_1 = dx_2
= 0.2$ for \eqref{eq:weakly-dispersive}.

The time discretization of both \eqref{eq:eps-wave} and
\eqref{eq:weakly-dispersive} is done via the second order Leap-Frog
method in two step formulation, e.g. for \eqref{eq:eps-wave} the
semi-discrete problem is thus
$$u^{(n+1)} = 2u^{(n)} -u^{(n-1)} +(dt)^2 \nabla \cdot (a^\eps\nabla u^{(n)}),$$
where $u^{(n)} \approx u^\eps(t = n\cdot dt)$. For the initialization
of the scheme we use the second order Taylor expansion
$$u^{(1)} = u^{(0)} + (dt)^2 \nabla \cdot (a^\eps\nabla u^{(0)}),$$
according to the initial condition $\partial_t u^\eps(x,0)=0.$ For
\eqref{eq:eps-wave} we use the time step $dt = \min \{0.01,
h_1/4,h_2/4\}$, for \eqref{eq:weakly-dispersive} we use $dt = 0.02$.

\subsubsection*{$\eps$-Convergence of the error}
\label{ssec.eps-conv}
Here, our aim is to determine experimentally the convergence rate of
the approximation error $\|u^\eps(t)-w^\eps(t)\|_{L^2(\Omega)}$.  To
this end we fix a periodic coefficient matrix field $a_Y$.  With the
identity matrix $I := \id_{\R^2}$ we set $a_Y(y) = \tilde a_Y(y)I$,
where $\tilde a_Y$ is defined through a rectangular geometry,
\begin{align}\label{E:a_rect}
  \tilde a_Y(y) = \frac{1}{2}+b(y)-\frac1{|Y|} \int_{Y}b(y) dy, \quad
  b(y) = \begin{cases} 1.6 \quad \text{for} \ 
    y\in [-\tfrac{11\pi}{13},\tfrac{11\pi}{13}]\times [-\tfrac{\pi}{3},\tfrac{\pi}{3}], \\
    0.2 \quad \text{otherwise},
\end{cases}
\end{align}
which is illustrated in Fig.~\ref{F:rect_compare} (a).

Our best numerical approximation of the effective coefficients is
$$a_1 \approx  0.281, \ a_2 \approx  0.179, \ \alpha_1 \approx -0.273,  \ 
\alpha_2 \approx -0.044, \ \beta \approx 0.024.$$ These values have
been computed with $h_1=2\pi/208$ and $h_2=2\pi/192$. Within rounding
to three decimal places the values do not change with a further mesh
refinement.

We solve \eqref{eq:eps-wave} and \eqref{eq:weakly-dispersive} for the
values $\eps = 0.2, 0.17, 0.15, 0.12, 0.1, 0.07$ up to times $t=12.5,
17.3, 22.5, 35, 50, 100 \approx \eps^{-2}/2$, respectively.  To keep
the computational expense within limits, we use the discretization
$h_1=2\pi\eps/13, h_2=2\pi\eps/12$ in the simulations of
\eqref{eq:eps-wave}. Using the same number of uniform elements in the
cell problems as in each periodic cell in \eqref{eq:eps-wave},
i.e. discretizing $Y$ by $2\times 13\times 12=312$ uniform elements,
we obtain the values
\begin{equation}\label{E:coeffs_rect}
  a_1  \approx  0.2784, \ a_2  \approx  0.1506, \  \alpha_1 \approx -0.369, \ 
  \alpha_2 \approx -0.034,  \ \beta \approx 0.032.
\end{equation}
We calculate solutions to the weakly dispersive equation
\eqref{eq:weakly-dispersive} with the values \eqref{E:coeffs_rect}
rather than with the converged coefficient values; this is justified
by the fact that \eqref{E:coeffs_rect} are the effective coefficients
for the particular discretization. Formulas \eqref{E:EF_sym} and the
values in \eqref{E:coeffs_rect} produce the effective coefficients
$$E_{11} = 1.3256, \  E_{22} =  0.2257, \ F_{1111}= F_{2222}= 0, \ 
F_{2121}= 0.1588, \ F_{1212}= 0.2957.$$ We can now investigate the
convergence of the error in $\eps$: The $L^2(\Omega)$-error
$\|u^\eps(t)-w^\eps(t)\|_{L^2(\Omega)}$ at the final time
$t=\eps^{-2}/2$ is plotted in Fig.~\ref{F:eps_conv}.  As
Fig.~\ref{F:eps_conv} (b) shows, the convergence is close to linear in
accordance with estimate \eqref{eq:approx-result}.
\begin{figure}[h!]
  \begin{center}
    \includegraphics[width=0.45\columnwidth]{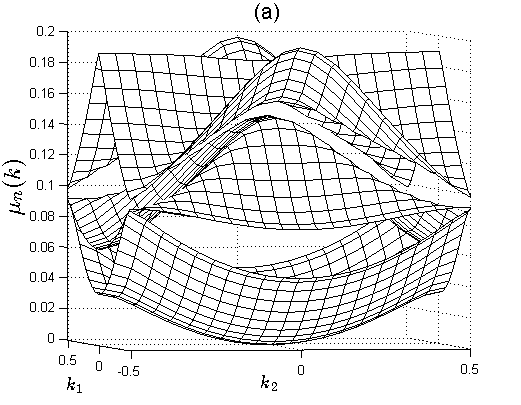}%
    \scalebox{0.6}{\includegraphics{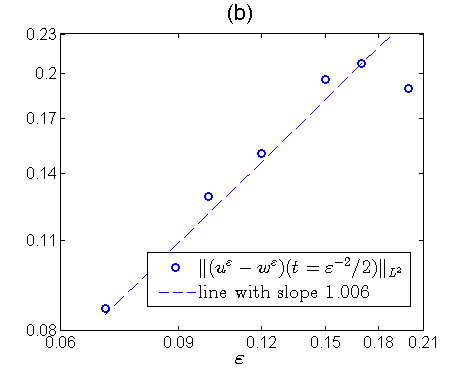}}
    \caption{\label{F:eps_conv}\small (a) The first three eigenvalues
      $\mu_0(k),\mu_1(k),\mu_2(k)$ of \eqref{eq:eigenvalue} for $a_Y$
      from \eqref{E:a_rect}. (b) Experimental convergence in $\eps$ of
      the error $e_\eps := \|(u^\eps-w^\eps)
      (t=\eps^{-2}/2)\|_{L^2(\Omega)}$. A logarithmic scale is used on
      both axes. The line with slope $1.1$ was obtained by a linear
      interpolation of the error in the logarithmic scale for the 5
      smallest $\eps$-values. The experimental rate is thus $e_\eps
      \sim \eps^{1.006}$.}
  \end{center}
\end{figure}

In Fig.~\ref{F:rect_compare} the solutions $u^\eps$ and $w^\eps$ are
plotted for $\eps=0.07$ at $t=100$. In both plots we see clearly the
main pulse located along an ellipse, we see the dispersive
oscillations behind the main pulse, and we see that the dispersion is
weakest along a ray that has an approximate angle $\pi/4$. Due to the
weak dispersion along this ray, the main pulse has its maximal
amplitude in this direction (compare also Fig.~\ref{F:rays_rect}).  An
excellent agreement of the two calculations is observed.
\begin{figure}[h!]
  \begin{center}
    \scalebox{0.49}{\includegraphics{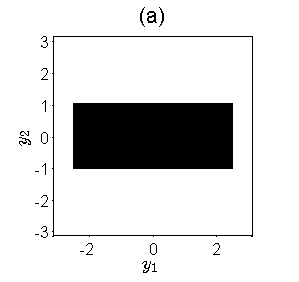}}
    \scalebox{0.48}{\includegraphics{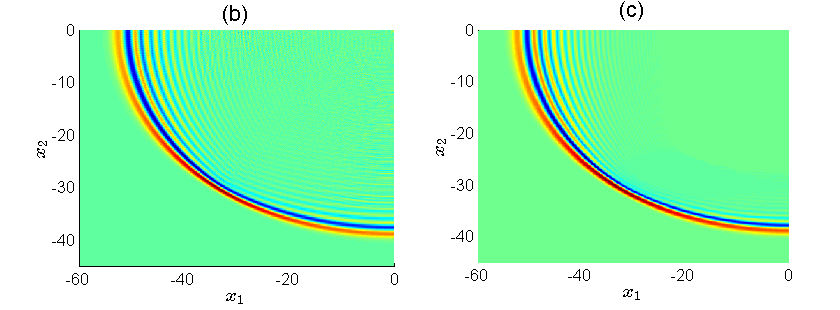}}
    \caption{\label{F:rect_compare} \small \small (a) Illustration of
      the geometry and of the function $a_Y(.)$ from
      \eqref{E:a_rect}. (b) The solution $u^\eps$ of the wave equation
      in a highly oscillatory medium. (c) The solution $w^\eps$ of the
      weakly dispersive wave equation with constant coefficients.  The
      geometry is as in \eqref{E:a_rect}, the plots are for
      $\eps=0.07$ and $t=100$.}
  \end{center}
\end{figure}

\subsubsection*{Further numerical examples}

We consider two more geometries. For a cross-shaped geometry as
illustrated in Fig.~\ref{F:crossB}, (a), we set $a_Y(y) :=
\tilde a_Y(y)I$ with
\begin{equation}\label{E:a_crossB}
  \tilde a_Y(y) = \begin{cases} 2 \quad &\text{for} \ 
    y\in [-\tfrac{7\pi}{9},\tfrac{7\pi}{9}]\times [-\tfrac{2\pi}{9},\tfrac{2\pi}{9}] 
    \cup  [-\tfrac{2\pi}{9},\tfrac{2\pi}{9}] \times [-\tfrac{7\pi}{9},\tfrac{7\pi}{9}], \\
    0.2 \quad &\text{otherwise}.
\end{cases}
\end{equation}
\begin{figure}[h!]
  \begin{center}
    \scalebox{0.52}{\includegraphics{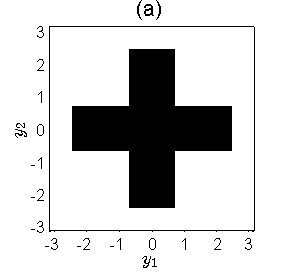}}
    \scalebox{0.52}{\includegraphics{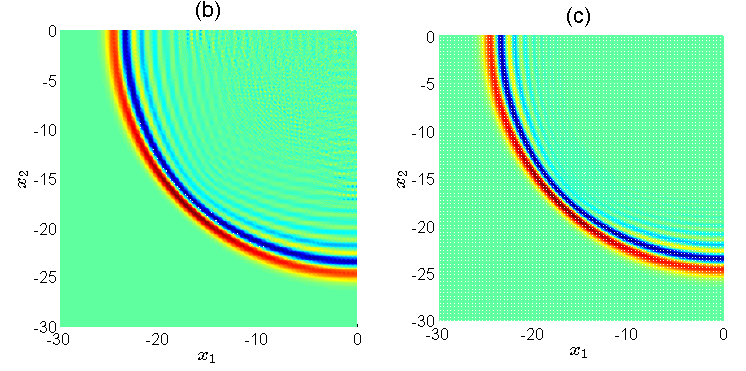}}
    \caption{\label{F:crossB} \small (a) Illustration of the geometry
      that defines the function $a_Y(.)$ in \eqref{E:a_crossB}. (b)
      The solution $u^\eps(x,t=40)$ for $\eps=0.07$. (c) The solution
      $w^\eps(x,t=40)$ for $\eps=0.07$.}
  \end{center}
\end{figure}

The additional symmetry $a_Y(y_1,y_2) = a_Y(y_2,y_1)$ for all $y\in Y$
implies the relations $a_1 = a_2$ and $\alpha_1 = \alpha_2$, see Lemma
2.6 in \cite{Bloch-DLS-2013}.  Algorithm \ref{A:AC} provides, with $Y$
discretized by $2\times 18^2 = 648$ uniform elements of size $h_1 =
h_2 = 2\pi/18$, the values
\begin{equation}\label{E:coef_cross}
  a_1=a_2  \approx  0.3816, \ 
  \alpha_1=\alpha_2 \approx  -0.1970, \ \beta \approx 0.0394.
\end{equation}
To check the accuracy, we calculated the values also using a fine
resolution with $2\times 360\times 360$ uniform elements. The fine
resolution provides $a_1=a_2\approx 0.406$, $\alpha_1=\alpha_2 \approx
-0.235, \beta\approx 0.044$, and the first three decimal places do not
change upon further refinement.  Similarly to the example in
Subsection \ref {ssec.eps-conv}, the discretization error in \eqref
{E:coef_cross} is quite large. Nevertheless, we use these coefficients
for the calculation of the solution $w^\eps$ of
\eqref{eq:weakly-dispersive}. The solution is compared to $u^\eps$,
which is computed with the corresponding spatial discretization $h_1 =
h_2 = 2\pi\eps /18$.  The results at time $t=40$ for $\eps=0.07$ are
plotted in Fig.~\ref{F:crossB}, (b) and (c).

\smallskip We finally consider a laminated structure $a_Y(y) = \tilde
a_Y(y) I$ with
\begin{equation}\label{E:a_lamin}
  \tilde a_Y(y) = \begin{cases} 2 \quad 
    &\text{for} \ y\in [-\pi,\pi]\times [-\tfrac{2\pi}{5},\tfrac{2\pi}{5}], \\
    0.2 \quad &\text{otherwise},
\end{cases}
\end{equation}
compare Fig.~\ref{F:lamin} (a).
\begin{figure}[h!]
  \begin{center}
    \scalebox{0.49}{\includegraphics{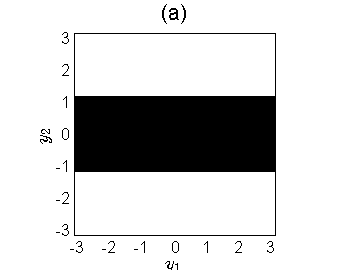}}
    \scalebox{0.48}{\includegraphics{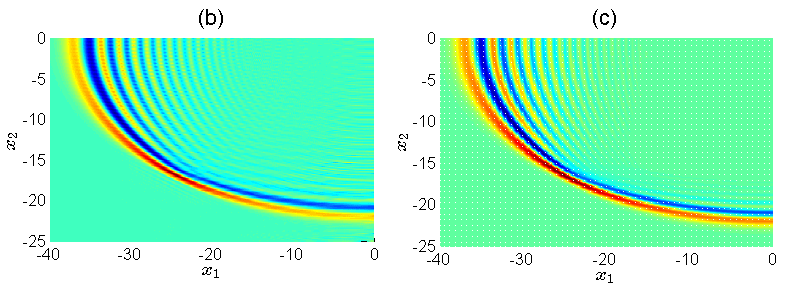}}
    \caption{\label{F:lamin} \small (a) Illustration of the function
      $a_Y(.)$ of \eqref{E:a_lamin}. (b) The solution $u^\eps(x,t=40)$
      for $\eps=0.07$. (c) The solution $w^\eps(x,t=40)$ for
      $\eps=0.07$.}
  \end{center}
\end{figure}
Effective coefficients are obtained by Algorithm \ref{A:AC}, $Y$ is
discretized with $2\times 12\times 16=384$ uniform elements
($h_1=2\pi/12$, $h_2=2\pi/16$):
\begin{equation}\label{E:coef_lamin}
  a_1  \approx  0.8750, \ a_2  \approx  0.3019, \ 
  \alpha_1 \approx -1.9185, \ \alpha_2 \approx -0.0933, \ \beta \approx 0.1448.
\end{equation}
The converged values (with four reliable digits, computed with
$2\times 140\times 180$ uniform elements) are: $a_1 \approx 0.9200,
a_2 \approx 0.3125$, $\alpha_1 \approx -1.9645$, $\alpha_2 \approx
-0.1170$, $\beta \approx 0.1599$.  The effective coefficients
determined using \eqref{E:EF_sym} and \eqref{E:coef_lamin} are
\begin{equation*}
  E_{11} = 2.1925,\ E_{22} = 0.3091, \ F_{1111}=F_{2222}= 0, \  
  F_{2121}= 0.7050, \ F_{1212}= 1.0964.
\end{equation*}
Equation \eqref{eq:eps-wave} was discretized in space using $h_1 =
2\pi\eps/12$ and $h_2 = 2\pi\eps/16$, the solutions $u^\eps$ and
$w^\eps$ are computed using the above coefficients; they are plotted
for $t=40$ and $\eps=0.07$ in Fig.~\ref{F:lamin}, (b) and (c).

\begin{figure}[ht]
  \begin{center}
    \scalebox{0.56}{\includegraphics{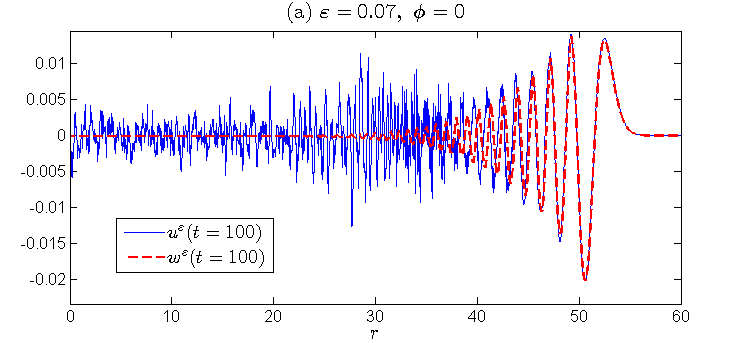}}\\
    \scalebox{0.6}{\includegraphics{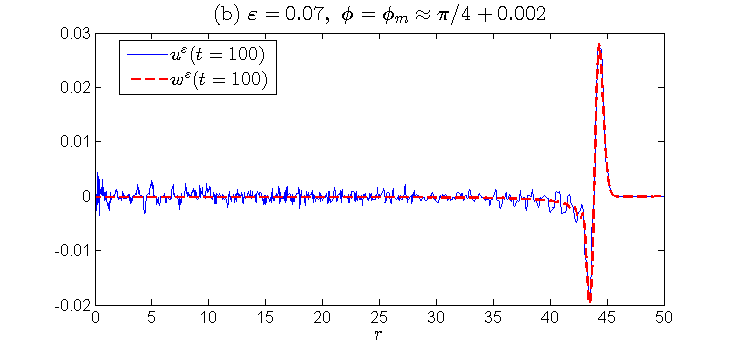}}
    \caption{\label{F:rays_rect} \small Solutions $u^\eps$ and
      $w^\eps$ for $\eps=0.07$ at $t=100$ for the rectangle geometry
      \eqref{E:a_rect} along two rays: (a) $\phi=0$; (b)
      $\phi=\phi_m\approx \pi/4 + 0.002$.}
  \end{center}
\end{figure}

\subsubsection*{Dependence of the dispersion on the propagation angle}

\begin{figure}[ht]
  \begin{center}
    \scalebox{0.37}{\includegraphics{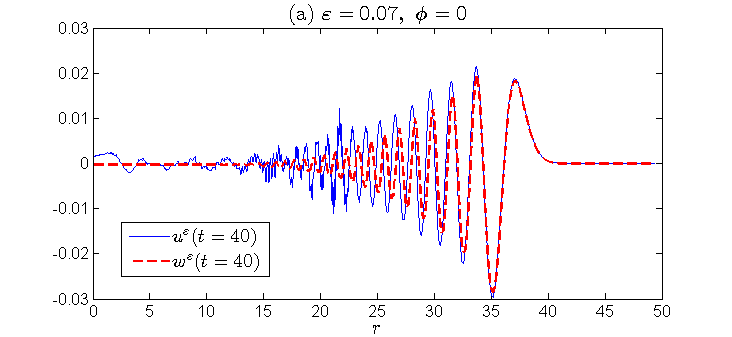}}
    \scalebox{0.37}{\includegraphics{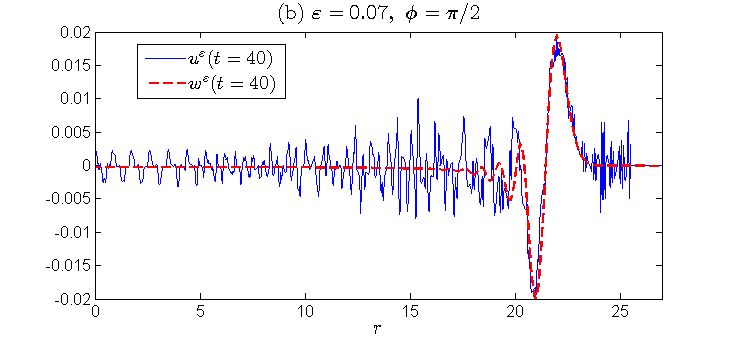}}
    \scalebox{0.39}{\includegraphics{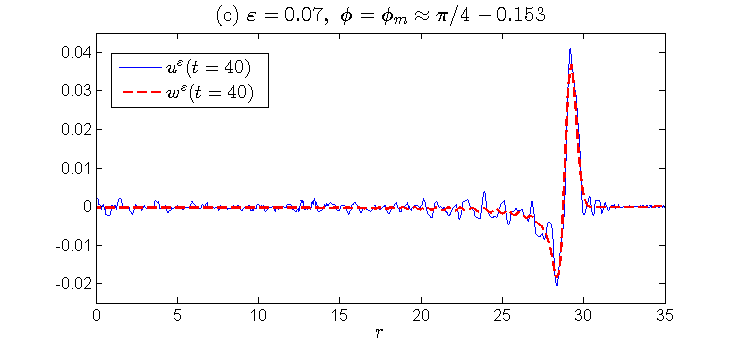}}
    \caption{\label{F:rays_lamin} \small Solutions $u^\eps$ and
      $w^\eps$ for $\eps=0.07$ at $t=40$ for the laminated geometry
      \eqref{E:a_lamin} along three rays: (a) $\phi=0$. (b) $\phi =
      \pi/2$. (c) $\phi = \phi_m \approx \pi/4 - 0.153$.  }
  \end{center}
\end{figure}

The rate of dispersion depends on the angle of propagation.  Here, we
have to distinguish the angle $\fhi$ of the elliptic coordinates and
the corresponding angle $\phi$ in polar coordinates (describing the
observable angle of the ray). We measure the angles such that $\phi =
0$ corresponds to the negative $x_1$-axis. The two angles are related
by $\tan(\phi) = \sqrt{a_2/a_1} \tan(\fhi)$. The dependence of the
dispersion on the angle can be observed in the numerical results: we
see few oscillations in a propagation angle of approximately $\phi =
\pi/4$, more oscillations along rays that are aligned with the
coordinate axes, i.e. at angles $\phi = 0$ and $\phi = \pi/2$. The
angle of minimal dispersion can be obtained by minimizing $\kappa =
\kappa(\fhi)$ of \eqref {eq:kappa-sincos}.

To illustrate the angular dependence, we consider the rectangular
geometry \eqref{E:a_rect} and the two rays corresponding to $\fhi = 0$
and $\fhi = \fhi_m$, the minimizer of $\kappa(.)$.  We plot $u^\eps$
and $w^\eps$ for $\eps=0.07$ and $t=100$ in
Fig.~\ref{F:rays_rect}. One can see clearly a much smaller dispersion
at the angle $\phi_m\approx \pi/4 + 0.002$ corresponding to
$\fhi_m$. Additionally, we observe a much larger error in the aligned
direction $\fhi = 0$. The values of the dispersion coefficient are
$\kappa(0) \approx -4.762$ and $\kappa(\fhi_m) \approx -0.175$.

For the laminate structure \eqref{E:a_lamin} we compare the solutions
along three directions in Fig.~\ref{F:rays_lamin}, namely the
horizontal direction $\phi = 0$ along which the structure is constant,
the vertical direction $\phi = \pi/2$, orthogonal to the laminates,
and an intermediate direction $\phi_m$ that corresponds to the
minimizer $\fhi_m \approx \pi/4 - 0.153$ of $\kappa$. The values of
the dispersion coefficient are $\kappa(0) \approx -2.506$,
$\kappa(\pi/2) \approx -1.024$, and $\kappa(\fhi_m) \approx 0.02$. The
analytical values of $\kappa$ are always non-positive, hence the
positive value $\kappa(\fhi_m)$ is caused by discretization errors.

\bibliographystyle{abbrv}
\bibliography{lit_bloch}

\end{document}